\documentclass[11pt]{article}

\usepackage{url}
\usepackage{amssymb,enumerate}
\usepackage{amsmath,amsfonts}
\usepackage{amsthm}
\usepackage{latexsym}
\usepackage{mathrsfs}
\usepackage{color}
\usepackage{microtype}
\usepackage{geometry}

\usepackage{algorithm}
\usepackage{algorithmicx}
\usepackage{algpseudocode}
\allowdisplaybreaks[3]
\usepackage{graphicx}
\usepackage{caption}
\usepackage{subfigure}
\usepackage{graphicx}
\usepackage{subcaption}
\usepackage{caption}
\usepackage{multirow}

\usepackage{amsmath,amsfonts,bm}

\def\1{\bm{1}}

\DeclareMathAlphabet{\mathsfit}{\encodingdefault}{\sfdefault}{m}{sl}
\SetMathAlphabet{\mathsfit}{bold}{\encodingdefault}{\sfdefault}{bx}{n}

\def\gO{{\mathcal{O}}}

\newcommand{\E}{\mathbb{E}}

\newcommand{\R}{\mathbb{R}}

\theoremstyle{plain}
\newtheorem{theorem}{Theorem}

\newtheorem{lemma}{Lemma}

\newtheorem{assumption}{Assumption}

\theoremstyle{definition}
\newtheorem{definition}{Definition}

\newtheorem{remark}{Remark}

\usepackage{graphicx, color}
\usepackage{algorithm, algpseudocode} 
\usepackage{mathrsfs} 

\usepackage{lipsum}

\title{Faster stochastic cubic regularized Newton methods with momentum}
\author{Yiming Yang \thanks{Department of Mathematics, Xiangtan University, China (email: {\tt yiming9780@outlook.com,pzheng@xtu.edu.cn}).} 
\and
Chuan He \thanks{Department of Mathematics, Link\"oping University, Sweden (email: {\tt chuan.he@liu.se}).}
\and
Xiao Wang \thanks{School of Computer Science and Engineering, Sun Yat-sen University, China (email: {\tt wangx936@mail.sysu.edu.cn}).}
\and
Zheng Peng $^*$
}

\date{
\today
}

\begin{document}
\maketitle

\begin{abstract}
Cubic regularized Newton (CRN) methods have attracted significant research interest because they offer stronger solution guarantees and lower iteration complexity. With the rise of the big-data era, there is growing interest in developing stochastic cubic regularized Newton (SCRN) methods that do not require exact gradient and Hessian evaluations. In this paper, we propose faster SCRN methods that incorporate gradient estimation with small, controlled errors and Hessian estimation with momentum-based variance reduction. These methods are particularly effective for problems where the gradient can be estimated accurately and at low cost, whereas accurate estimation of the Hessian is expensive. Under mild assumptions, we establish the iteration complexity of our SCRN methods by analyzing the descent of a novel potential sequence. Finally, numerical experiments show that our SCRN methods can achieve comparable performance to deterministic CRN methods and vastly outperform first-order methods in terms of both iteration counts and solution quality.

\end{abstract}


\medskip


\section{Introduction}\label{sec:intro}
In this paper, we consider the smooth unconstrained optimization problem
\begin{align}\label{ucpb}
\min_{x\in\R^n}f(x),
\end{align}
where $f:\R^n\rightarrow \R$ is twice continuously differentiable. We assume that problem \eqref{ucpb} has at least one optimal solution. Over the past few years, second-order methods have gained popularity for handling problem \eqref{ucpb} due to their ability to converge in fewer iterations than first-order methods and to deliver higher solution quality. However, the computational overhead incurred per evaluation of the Hessian matrix hinders the scalability of second-order methods in modern large-scale settings. To better leverage second-order information in these settings, this paper aims to propose practical second-order methods---particularly variants of the cubic regularized Newton (CRN) method---to solve problem \eqref{ucpb} and analyze their iteration complexity for finding an approximate second-order stationary point (SOSP) of \eqref{ucpb}.


Second-order methods have recently received considerable attention for their strong solution guarantees and rapid convergence, with substantial progress made in designing new second-order methods with complexity guarantees for solving problem \eqref{ucpb}. In particular, CRN methods \cite{agarwal2017finding,carmon2019gradient,cartis2011adaptive,nesterov2006cubic}, trust-region methods \cite{curtis2021trust,curtis2017trust,martinez2017cubic}, second-order line-search method \cite{royer2018complexity}, inexact regularized Newton method \cite{curtis2019inexact}, quadratic regularization method \cite{birgin2017use}, and Newton-CG methods \cite{he2025newton,he2023newton,royer2020newton} were developed for finding an $(\epsilon,\sqrt{\epsilon})$-SOSP $x$ of problem \eqref{ucpb} satisfying 
\begin{align*}
\|\nabla f(x)\|\le \epsilon,\qquad \lambda_{\min}(\nabla^2 f(x)) \ge -\sqrt{\epsilon},
\end{align*}
where $\epsilon\in(0,1)$ is a tolerance parameter and $\lambda_{\min}(\cdot)$ denotes the minimum eigenvalue of the associated matrix. Under suitable assumptions, it was shown that these methods achieve an iteration complexity of $\mathcal{O}(\epsilon^{-3/2})$ for finding an $(\epsilon,\sqrt{\epsilon})$-SOSP, which has been proven to be optimal in \cite{carmon2020lower,cartis2018worst}. Besides, several gradient-based methods with random perturbations (e.g., \cite{allen2018neon2, jin2018accelerated, xu2017neon+}) have also been developed to find an $(\epsilon,\sqrt{\epsilon})$-SOSP with high probability.

Despite the significant advantages of second-order methods, their high per-iteration cost limits their use in large-scale problems. To address this limitation, many research works have focused on developing inexact and stochastic variants of second-order methods. In particular, inexact and stochastic versions of CRN methods \cite{bellavia2021adaptive,cartis2011adaptive,doikov2025first,grapiglia2022cubic,kohler2017sub,tripuraneni2018stochastic,wang2019note,xu2020newton}, trust-region method \cite{xu2020newton}, Newton-CG method \cite{yao2023inexact}, and subsampling line-search method \cite{bergou2022subsampling}, have been developed to seek an approximate SOSP of problem~\eqref{ucpb}. These methods achieve a similar order of complexity bounds as their exact variants. 
Yet, in each iteration, these methods require a fairly accurate approximation of the gradient and Hessian with small errors depending on the desired tolerance $\epsilon$ (potentially after a certain number of iterations), which remains quite restrictive in practice.

Furthermore, several recent works \cite{chayti2024improving,jaggiunified,wang2019stochastic,zhang2022adaptive,zhou2019stochastic} have developed stochastic second-order methods leveraging variance reduction techniques. These methods apply variance reduction to iteratively correct estimation errors using stochastic derivative information from previous iterations, thereby avoiding the need to construct accurate derivative estimates at each step. Such desirable features enable these methods to maintain low per-iteration costs, making them more practical for large-scale problems. Among these works, \cite{chayti2024improving} is the only one that do not assuming a finite-sum structure for the objective function. This work proposed two methods that achieve iteration complexity bounds of $\mathcal{O}(\epsilon^{-7/2})$ and $\mathcal{O}(\epsilon^{-10/3})$, respectively, for finding an approximate stochastic stationary point $x$ satisfying $\mathbb{E}[\|\nabla f(x)\|] \le \epsilon$. However, these complexity bounds leave a significant gap compared to the bound of $\mathcal{O}(\epsilon^{-3/2})$ achieved by the deterministic CRN method; moreover, they are worse than the iteration complexity of $\mathcal{O}(\epsilon^{-3})$ achieved by stochastic first-order methods (e.g., \cite{cutkosky2019momentum,fang2018spider,li2021page}).

Motivated by the aforementioned discussions, we aim to rethink the design of stochastic second-order methods and identify the types of stochastic optimization problems for which they are preferable. In this paper, we focus on a class of problems where the gradient can be estimated relatively easily with small errors (see Assumption \ref{asp:basic}(c)), but estimating the Hessian is more expensive; therefore, only unbiased stochastic Hessian estimators with bounded moments (see Assumption \ref{asp:basic}(d)) are available at each step. Specifically, we propose two new variants of SCRN methods for solving such problems. Under mild assumptions, we establish their iteration complexity for finding an $(\epsilon_g,\epsilon_H)$-stochastic second-order stationary point (SSOSP) of problem \eqref{ucpb} (see Definition \ref{def:ssosp}), based on an analysis of the descent of a novel potential sequence (see \eqref{def:pot}). For ease of comparison, we summarize the iteration complexity, the number of samples per iteration, smoothness conditions, and stationary measures for vanilla gradient descent, stochastic first-order methods, CRN, variants of SCRN, and our methods for nonconvex optimization in Table \ref{table:sum-ic}.

\begin{table}[htbp]
\centering
\caption{Comparison of vanilla gradient descent (GD), stochastic gradient methods with momentum (SGD-M), CRN, SCRN, and SCRN with momentm (SCRN-M) in terms of iteration complexity, the number of samples per iteration, smoothness conditions, and stationary measures.}
\smallskip
\renewcommand{\arraystretch}{1.2} 
\resizebox{\textwidth}{!}{
\begin{tabular}{c|c|c|c|c}
\hline
\multicolumn{5}{c}{First-order methods}\\
\hline
& \multirow{2}{*}{iteration complexity} & gradient samples & \multirow{2}{*}{smoothness condition} & \multirow{2}{*}{stationary measure}\\
&  & per iteration & & \\
\hline
GD & $\mathcal{O}(\epsilon^{-2})$ & --- & $\nabla f$ Lipschitz & $\|\nabla f(x)\|\le\epsilon$ \\
SGD \cite{ghadimi2013stochastic} & $\mathcal{O}(\epsilon^{-4})$ & 1 & $\nabla f$ Lipschitz & $\E[\|\nabla f(x)\|^2]\le\epsilon^2$ \\
SGD-M \cite{cutkosky2020momentum,he2025complexity} & $\widetilde{\mathcal{O}}(\epsilon^{-{(3p+1)/p}})$ & 1 & $\mathcal{D}^p f$ Lipschitz\footnotemark & $\E[\|\nabla f(x)\|]\le\epsilon$ \\
SGD-M \cite{cutkosky2019momentum} & $\mathcal{O}(\epsilon^{-3})$  & 1 & $G$ average Lipschitz & $\E[\|\nabla f(x)\|^2]\le\epsilon^2$ \\
\hline
\multicolumn{5}{c}{Second-order methods}\\
\hline
& \multirow{2}{*}{iteration complexity} & Hessian samples & \multirow{2}{*}{smoothness condition} & \multirow{2}{*}{stationary measure}\\
&  & per iteration & & \\
\hline
CRN \cite{nesterov2006cubic} & $\mathcal{O}(\max\{\epsilon_g^{-3/2},\epsilon_H^{-3}\})$ & --- & $\nabla^2 f$ Lipschitz & $\|\nabla f(x)\|\le\epsilon_g$, $\lambda_{\min}(\nabla^2 f(x))\ge-\epsilon_H$ \\
SCRN \cite{tripuraneni2018stochastic} & $\mathcal{O}(\epsilon^{-3/2})$ & $\widetilde{\mathcal{O}}(\epsilon^{-1})$ & $\nabla f, \nabla^2 f$ Lipschitz & $\|\nabla f(x)\|\le\epsilon$, $\lambda_{\min}(\nabla^2 f(x))\ge-\sqrt{\epsilon}\ \ {\mathrm{w.h.p.}}$ \\
SCRN-M \cite{chayti2024improving} & $\mathcal{O}(\epsilon^{-7/2})$  & 1 & $\nabla^2 f$ Lipschitz & $\E[\|\nabla f(x)\|^{3/2}]\le\epsilon^{3/2}$ \\
Algorithm \ref{alg:unf-ssom-pm} (ours) &  $\mathcal{O}(\max\{\epsilon_g^{-7/4},\epsilon_H^{-7}\})$  & 1 & $\nabla^2 f$ Lipschitz & {$\E[\|\nabla f(x)\|^{3/2}]\le\epsilon_g^{3/2}$, $\E[\lambda_{\min}(\nabla^2 f(x))^3]\ge-\epsilon_H^3$} \\
Algorithm \ref{alg:unf-ssom-rm} (ours) & $\mathcal{O}(\max\{\epsilon_g^{-5/3}, \epsilon_H^{-5}\})$  & 1 & $H$ average Lipschitz & {$\E[\|\nabla f(x)\|^{3/2}]\le\epsilon_g^{3/2}$, $\E[\lambda_{\min}(\nabla^2 f(x))^3]\ge-\epsilon_H^3$} \\
\hline
\end{tabular}
}
\label{table:sum-ic}
\end{table}

The main contributions of this paper are highlighted below.
\begin{itemize}
    \item We propose two new SCRN methods (Algorithms \ref{alg:unf-ssom-pm} and \ref{alg:unf-ssom-rm}), which adopt stochastic gradients with small errors and incorporate momentum-based variance reduction for estimating Hessian. Under mild assumptions, we establish their iteration complexity based on an analysis of the descent of a novel potential sequence. To the best of our knowledge, the obtained complexity bounds are new to the literature.

    \item We conduct numerical experiments (Section \ref{sec:experiments}) to compare our SCRN methods with deterministic CRN, other SCRN variants, and first-order methods. The numerical results show that our methods achieve performance comparable to deterministic CRN methods and significantly outperform other SCRN variants and first-order methods.

\end{itemize}

The remainder of this paper is organized as follows. In Section \ref{sec:nap}, we introduce the notation and assumptions used throughout the paper. Sections \ref{sec:scn-pm} and \ref{sec:scn-rm} present two new variants of SCRN and establish their iteration complexity. Section \ref{sec:experiments} reports preliminary numerical results. Finally, Section~\ref{sec:proof} provides the proofs of the main results.

\footnotetext{$G$ and $\mathcal{D}^p f$ represent the stochastic gradient and the $p$th-order derivative of $f$, respectively.}

\section{Notation and assumptions}\label{sec:nap}
Throughout this paper, we use $\R^n$ to denote the $n$-dimensional Euclidean space endowed with the standard inner product $\langle\cdot,\cdot\rangle$. We let $\|\cdot\|$ denote the Euclidean norm for vectors and the spectral norm for matrices, and use $\|\cdot\|_F$ to denote the Frobenius norm for matrices. For a given matrix $H\in\R^{n\times n}$, we use $\lambda_{\min}(H)$ to denote its minimum eigenvalue, and use $\mathrm{Tr}(H)$ to denote the trace of $H$. We let $I$ be the $n\times n$ identity matrix. In addition, we use $\widetilde{\gO}(\cdot)$ to represent $\gO(\cdot)$ with polylogarithmic terms omitted.


We now make the following assumptions throughout this paper.

\begin{assumption}\label{asp:basic}
\begin{enumerate}[{\rm (a)}]
\item There exists a finite $f_{\mathrm{low}}$ such that $f(x)\ge f_{\mathrm{low}}$ for all $x\in\R^n$.
\item There exist $L>0$ and $L_F>0$ such that $\|\nabla^2 f(y) - \nabla^2 f(x)\|\le L\|y-x\|$ and $\|\nabla^2 f(y) - \nabla^2 f(x)\|_F\le L_F\|y-x\|$ hold for all $x,y\in\R^n$.
\item For any $\delta\in(0,1)$, we have access to a stochastic gradient estimator $G_{\delta}:\mathbb{R}^n\times \mathcal{Z}\to\R^n$ satisfying 
\begin{align}\label{ineq:G-var}
\E_\zeta[\|G_\delta(x;\zeta) - \nabla f(x)\|^{3/2}] \le \delta^{3/2} \qquad\forall x\in\R^n.    
\end{align}
\item We have access to a stochastic Hessian estimator $H:\mathbb{R}^n\times\Xi\to\R^{n \times n}$ satisfying 
\begin{align}\label{ineq:H-Variance}
\mathbb{E}_\xi[H(x;\xi)] = \nabla^2 f(x),\quad \mathbb{E}_\xi[\|H(x;\xi) - \nabla^2 f(x)\|_F^3] \le \sigma^3\qquad\forall x\in\R^n   
\end{align}
for some $\sigma>0$.
\end{enumerate}    
\end{assumption}

\begin{remark}
{\it (i)
Assumptions \ref{asp:basic}(a) and \ref{asp:basic}(b) are common in the literature on CRN methods (e.g., see \cite{nesterov2006cubic,zhang2022adaptive}). It follows from Assumption \ref{asp:basic}(b) that
\begin{align}\label{ineq:1st-desc}
&\|\nabla f(y) - \nabla f(x) - \nabla^2 f(x)(y-x)\| \le \frac{L}{2}\|y-x\|^2 \quad \forall x,y\in\R^n,\\
&f(y) \le f(x) + \nabla f(x)^T(y-x) + \frac{1}{2}(y-x)^T\nabla^2f(x)(y-x) + \frac{L}{6}\|y-x\|^3\quad \forall x,y\in\R^n.\label{ineq:2st-desc}
\end{align}
In addition, Assumption \ref{asp:basic}(c) states that $G_\delta(\cdot;\xi)$ approximates the true gradient $\nabla f(\cdot)$ to any desired accuracy in expectation, while Assumption \ref{asp:basic}(d) states that the stochastic Hessian $H(\cdot;\xi)$ is an unbiased estimator of $\nabla^2 f(\cdot)$ and has a bounded third-order central moment. 

(ii) We made two Lipschitz continuity assumptions on $\nabla^2 f$ in Assumption \ref{asp:basic}(b). In particaulr, the Lipschitz continuity with respect to the spectral norm is used to estimate the reduction of function values at each iteration of our SCRN methods (see  Lemma \ref{lem:ppt-desc} below and the classical analysis in \cite{nesterov2006cubic}). In comparison, the Lipschitz continuity with respect to the Frobenius norm is used to derive a recursive relation for the decreasing estimation error given by the momentum update (see Lemmas \ref{lem:rec-Hes-pm} and \ref{lem:rec-Hes-rm}). In our complexity analysis, we found that within the Schatten family of matrix norms, only the Frobenius norm seems effective for analyzing Hessian estimation error with momentum updates, while other norms, such as the spectral and nuclear norms, do not appear to be useful. Our explanation is that for any Schatten-$p$ norm $\|\cdot\|_{S_p}$, our analysis requires the norm $\|\cdot\|_{S_p}$ to be continuously differentiable. However, this condition is satisfied only when $p=2$, which corresponds to the Frobenius norm.
}
\end{remark}

We next introduce the definition of an approximate SSOSP, which our methods aim to achieve.
\begin{definition}\label{def:ssosp}
For any $\epsilon_g,\epsilon_H\in(0,1)$, we say that $x\in\R^n$ is an $(\epsilon_g,\epsilon_H)$-stochastic second-order stationary point (SSOSP) of problem~\eqref{ucpb} if it satisfies $\E[\|\nabla f(x)\|^{3/2}] \le\epsilon_g^{3/2}$ and $\E[\lambda_{\min}(\nabla^2 f(x))^3]\ge-\epsilon_H^3$.
\end{definition}

\section{SCRN with Polyak momentum}\label{sec:scn-pm}

In this section, we propose an SCRN method with Polyak momentum, and then establish its iteration complexity under Assumption \ref{asp:basic}.

Specifically, our SCRN method with Polyak momentum generate three sequences, $\{g^k\}$, $\{M_k\}$, and $\{x^k\}$. At the $k$th iteration, this method first computes $g^k$ as a stochastic gradient of $f$ at $x^k$ with error $\delta_k$, and then computes $M_k$ as a weighted average of the stochastic Hessians evaluated at $x^0,\ldots,x^k$. The next iterate $x^{k+1}$ is obtained by solving a cubic regularized Newton subproblem. Details of this method are described in Algorithm~\ref{alg:unf-ssom-pm}, with a specific choice of input parameters given in Theorem \ref{th:complexity-pm-c}.

\begin{algorithm}[!htbp]
\caption{SCRN with Polyak momentum}
\label{alg:unf-ssom-pm}
\begin{algorithmic}[0]
\State \textbf{Input:} starting point $x^0\in\R^n$, regularization parameters $\{\eta_k\}\subset(0,\infty)$, error parameters $\{\delta_k\}\subset(0,1)$, momentum parameters $\{\theta_k\}\subset(0,1)$.
\State \textbf{Initialize:} $M_{-1}=0$ and $\theta_{-1}=1$.
\For{$k=0,1,2,\ldots$}
\State Construct the gradient and Hessian estimators:
\begin{align}\label{update-Mk-pm}
g^k = G_{\delta_k}(x^k;\zeta^{k}),\quad M_k = (1 - \theta_{k-1}) M_{k-1} + \theta_{k-1} H(x^k;\xi^{k}).
\end{align}
\State Update the next iterate:
\begin{align*}
x^{k+1} \in \underset{x \in \mathbb{R}^{n}}{\mathrm{Arg\,min}} \Big\{ (g^k)^T (x-x^k) + \frac{1}{2}(x-x^k)^TM_k(x-x^k) + \frac{1}{6\eta_k}\|x-x^k\|^3\Big\}.
\end{align*}
\EndFor
\end{algorithmic}
\end{algorithm}

The following theorem establishes the iteration complexity of Algorithm \ref{alg:unf-ssom-pm} for computing an $(\epsilon_g,\epsilon_H)$-SSOSP of problem \eqref{ucpb}. Its proof is provided in Section \ref{subsec:proof-pm}.

\begin{theorem}\label{th:complexity-pm-c}
Suppose that Assumption \ref{asp:basic} holds. Define 
\begin{align}\label{def:Mc}
M_{\mathrm{pm}}:= 54\Big(f(x^0)-f_{\mathrm{low}} + \sigma^3L_{F}^{-2} + L_{F}^{3/2}\sigma^3 + 1\Big),
\end{align}
where $f_{\mathrm{low}}$, $L_F$, and $\sigma$ be given in Assumption \ref{asp:basic}. Let $\{x^{k}\}$ be all iterates generated by Algorithm \ref{alg:unf-ssom-pm} with input parameters $\{(\eta_k,\theta_k,\delta_k)\}$ given by
\begin{align}\label{c-para-pm}
\eta_k=\frac{1}{9K^{2/7}},\quad \theta_k=\frac{7L_F}{3K^{2/7}}, \quad \delta_k=\frac{1}{9K^{4/7}},\quad\forall k\ge0.
\end{align}
Then, for any $\epsilon_g,\epsilon_H\in(0,1)$, $x^{\iota_K}$ is an $(\epsilon_g,\epsilon_H)$-SSOSP of problem \eqref{ucpb} for all $K$ satisfying
\begin{align}\label{stat-fix-pm}
K\ge \max\bigg\{\Big(\frac{(3M_{\mathrm{pm}})^{2/3}}{\epsilon_g}\Big)^{7/4},\Big(\frac{(108M_{\mathrm{pm}})^{1/3}}{\epsilon_H}\Big)^7,\Big(\frac{2L}{9}\Big)^{7/2},\Big(\frac{7L_F}{3}\Big)^{7/2},1\bigg\},
\end{align}
where $\iota_K$ is uniformly drawn from $\{1,\ldots,K\}$. 
\end{theorem}

\begin{remark}
{\it From Theorem \ref{th:complexity-pm-c}, we see that Algorithm \ref{alg:unf-ssom-pm} with input parameters given by \eqref{c-para-pm} achieves an iteration complexity of $\mathcal{O}(\max\{\epsilon_g^{-7/4},\epsilon_H^{-7}\})$ for finding an $(\epsilon_g,\epsilon_H)$-SSOSP of problem \eqref{ucpb}.}
\end{remark}

\section{SCRN with recursive momentum}\label{sec:scn-rm}
In this section, we propose an SCRN method with recursive momentum, and then establish its iteration complexity. 

Specifically, our SCRN method with recursive momentum generate three sequences, $\{g^k\}$, $\{M_k\}$, and $\{x^k\}$. At the $k$th iteration, this method first compute $g^k$ as a stochastic gradient of $f$ at $x^k$ with error $\delta_k$, and compute $M_k$ as a weighted average of the stochastic Hessian evaluated at $x^0,\ldots,x^k$ using the recursive momentum scheme proposed in \cite{cutkosky2019momentum}. The next iterate $x^{k+1}$ is obtained by solving a cubic regularized Newton subproblem. Details of this method are provided in Algorithm \ref{alg:unf-ssom-rm}, with a specific choice of input parameters given in Theorem \ref{th:complexity-rm-c}.

Before analyzing Algorithm \ref{alg:unf-ssom-rm}, we make the following assumption regarding the {\it mean-cubed smoothness} of the stochastic Hessian estimator $H(\cdot;\xi)$.
\begin{assumption}\label{asp:mean-squared}
There exists $L_H>0$ such that $\mathbb{E}_\xi[\|H(y;\xi) - H(x;\xi)\|_F^3]\le L^3_H\|y-x\|^3_F$ holds for all $x,y\in\R^n$.
\end{assumption}

\begin{algorithm}[!htbp]
\caption{SCRN with recursive momentum}
\label{alg:unf-ssom-rm}
\begin{algorithmic}[0]
\State \textbf{Input:} starting point $x^0\in\R^n$, regularization parameters $\{\eta_k\}\subset(0,\infty)$, error control parameters $\{\delta_k\}\subset(0,1)$, momentum parameters $\{\theta_k\}\subset(0,1)$.
\State Initialize: $M_{-1}=0$ and $\theta_{-1}=1$.
\For{$k=0,1,2,\ldots$}
\State Construct the gradient and Hessian estimators:
\begin{align}\label{update-Mk-rm}
g^k = G_{\delta_k}(x^k;\zeta^{k}),\quad M_k = (1-\theta_{k-1})M_{k-1} + H(x^{k};\xi^{k}) - (1-\theta_{k-1})H(x^{k-1};\xi^{k}).
\end{align}
\State Update the next iterate:
\begin{align*}
x^{k+1} \in \underset{x \in \mathbb{R}^{n}}{\mathrm{Arg\,min}} \Big\{(g^k)^T (x-x^k) + \frac{1}{2}(x-x^k)^TM_k(x-x^k) + \frac{1}{6\eta_k}\|x-x^k\|^3\Big\}.
\end{align*}
\EndFor
\end{algorithmic}
\end{algorithm}

The following theorem establishes an iteration complexity bound of Algorithm \ref{alg:unf-ssom-rm} for computing an $(\epsilon_g,\epsilon_H)$-SSOSP of problem \eqref{ucpb}. Its proof is relegated to Section \ref{subsec:proof-rm}. 

\begin{theorem}\label{th:complexity-rm-c}
Suppose that Assumptions \ref{asp:basic} and \ref{asp:mean-squared} hold. Define 
\begin{align}\label{def:Mct}
M_{\mathrm{rm}} := 75(f(x^0)-f_{\mathrm{low}} + \sigma^3(L_{F}^3 + L_H^3)^{-2/3} + (L_{F}^3+L_H^3)\sigma^3 + 1),
\end{align}
where $f_{\mathrm{low}}$, $L_F$, and $\sigma$ are given in Assumption \ref{asp:basic}, and $L_H$ is given in Assumption \ref{asp:mean-squared}. Let $\{x^{k}\}$ be all iterates generated by Algorithm \ref{alg:unf-ssom-rm} with input parameters $\{(\eta_k,\theta_k,\delta_k)\}$ given by
\begin{align}\label{c-para-rm}
\eta_k= \frac{1}{17K^{1/5}},\quad \theta_k= \frac{625(L_{F}^3 + L_H^3)^{2/3}}{289K^{2/5}},\quad \delta_k= \frac{1}{17K^{3/5}}\quad \forall k\ge0.
\end{align}
Then, for any $\epsilon_g,\epsilon_H\in(0,1)$, $x^{\iota_K}$ is an $(\epsilon_g,\epsilon_H)$-SSOSP of problem \eqref{ucpb} for all $K$ satisfying
\begin{align}\label{stat-fix-rm}
K\ge \max\Big\{\Big(\frac{(3M_{\mathrm{rm}})^{2/3}}{\epsilon_g}\Big)^{5/3},\Big(\frac{(281M_{\mathrm{rm}})^{1/3}}{\epsilon_H}\Big)^5,\Big(\frac{2L}{17}\Big)^{5},7(L_{F}^3 + L_H^3)^{5/3},1\Big\},
\end{align}
where $\iota_K$ is uniformly drawn from $\{1,\ldots,K\}$. 
\end{theorem}

\begin{remark}
{\it From Theorem \ref{th:complexity-rm-c}, we observe that Algorithm \ref{alg:unf-ssom-rm} with input parameters given by \eqref{c-para-rm} achieves an iteration complexity of $\mathcal{O}(\max\{\epsilon_g^{-5/3},\epsilon_H^{-5}\})$ for finding an $(\epsilon_g,\epsilon_H)$-SSOSP of problem \eqref{ucpb}. This bound improves upon the one for Algorithm \ref{alg:unf-ssom-pm} established in Theorem \ref{th:complexity-pm-c}.}
\end{remark}

\section{Numerical experiments}\label{sec:experiments}

In this section, we conduct numerical experiments to evaluate the performance of Algorithms \ref{alg:unf-ssom-pm} and \ref{alg:unf-ssom-rm}, abbreviated as SCRN-PM and SCRN-RM, respectively. We compare these methods with the adaptive cubic regularized Newton method \cite{cartis2011adaptive} (A-CRN), stochastic cubic regularized Newton method with momentum \cite{chayti2024improving} (SCRN-M), and SpaRSA \cite{wright2009sparse}. The experiments are conducted on three nonconvex statistical learning problems using datasets from LIBSVM\footnote{\url{https://www.csie.ntu.edu.tw/~cjlin/libsvmtools/datasets/}}. All the algorithms are coded in Python, and all computations are performed on a laptop with an Intel Core i7 processor and 10 GB of RAM.



\subsection{Regularized logistic regression problems}\label{subsec:5.1}
In this subsection, we consider the regularized logistic regression problem:
\begin{align}\label{logistic-pro}
\min_{x\in\R^n}\ \sum_{i=1}^m \left(b_i \ln(\phi(x^Ta_i)) + (1 - b_i) \ln(1 - \phi(x^Ta_i) \right) + \lambda \sum_{j=1}^n\frac{(\gamma x_j)^2}{1 + (\gamma x_j)^2},
\end{align}
where $\phi(t) = e^t/(1 + e^t)$ denotes the sigmoid function, $\{(a_{i},b_{i})\}_{1\le i\le m} \subset \R^n\times \R$ is the given data, and $(\lambda,\gamma)=(0.001,10)$. We consider three datasets `a9a', `phishing', and `w8a' from LIBSVM. 

We apply SCRN-PM, SCRN-RM, A-CRN, SCRN-M, and SpaRSA to solve problem \eqref{logistic-pro}. All methods are initialized at $[0.5,\ldots,0.5]^T$. For SCRN-M, we choose $50\%$ of the elements from the gradient and Hessian, respectively, to construct unbiased estimators of $\nabla f$ and $\nabla^2 f$. For SCRN-PM and SCRN-RM, we choose set $g^k$ as full gradients $\nabla f(x^k)$ for all $k\ge0$, and choose $50\%$ of the elements from the Hessian to construct unbiased Hessian estimators. For CRN and all SCRN methods, we adopt the Lanczos method used in \cite{cartis2011adaptive} to solve the cubic regularized subproblems. We compare these methods in terms of the function value gap defined by $f(x^k)-f^*$, where $f^*$ is the minimum objective value found during the first $2000$ iterations across all methods. The algorithmic parameters are selected to suit each method well in terms of computational performance.

For each dataset, we plot the function value gap in Figure \ref{logistic} to illustrate the convergence behavior of all competing methods. From Figure \ref{logistic}, we observe that SCRN-PM and SCRN-RM vastly outperform SCRN-M and SpaRSA. In addition, SCRN-PM and SCRN-RM achieve a comparable performance to CRN in terms the number of iterations, while outperforming CRN in terms of CPU time. These observations indicate that full gradients significantly improve the performance of the SCRN algorithm, bringing it close to that of the deterministic CRN while reducing computation time. However, when SCRN uses stochastic gradients, its convergence becomes much slower and may offer little to no advantage over first-order methods. Furthermore, we observe that SCRN-RM slightly outperforms SCRN-PM, which corroborates our theoretical results.



\begin{figure*}[htbp]
  	\centering
  	\subfigure
  	{\includegraphics[width=0.327\textwidth]{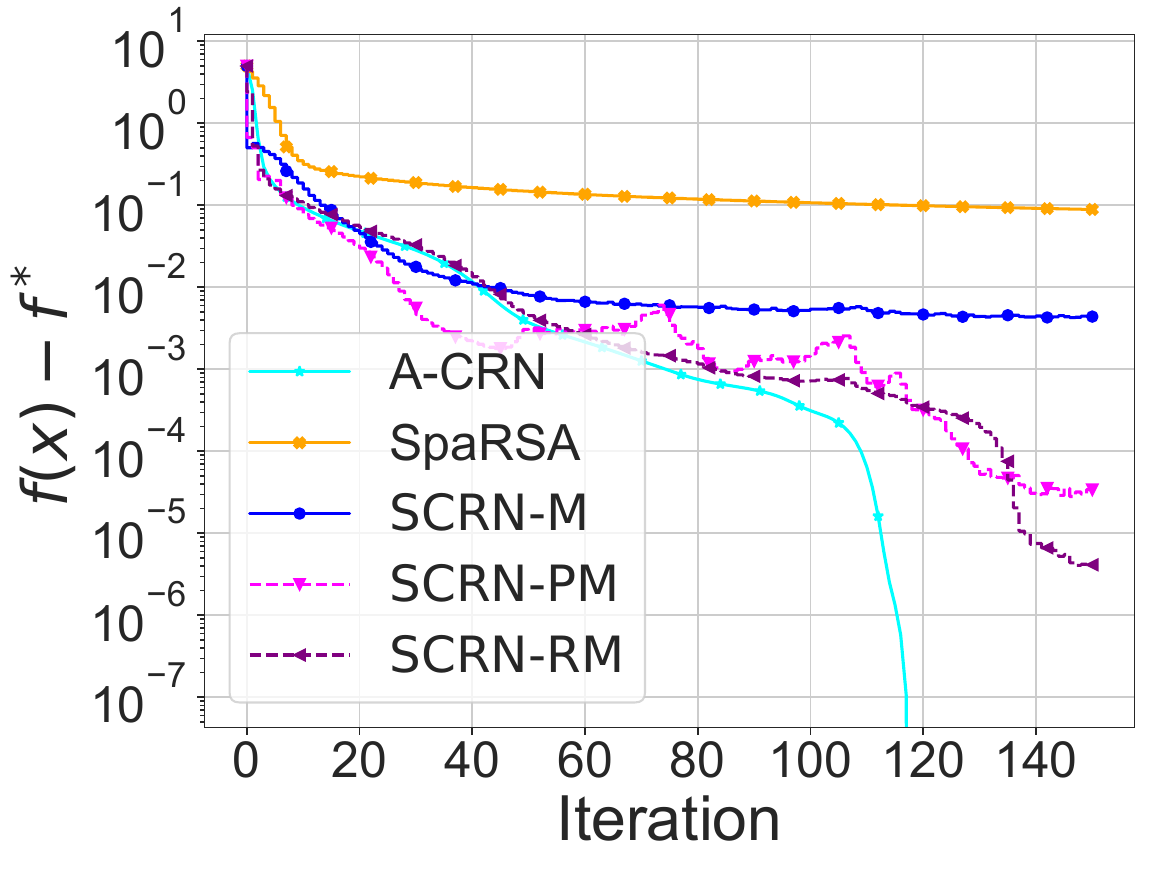}}
  	\subfigure
  	{\includegraphics[width=0.327\textwidth]{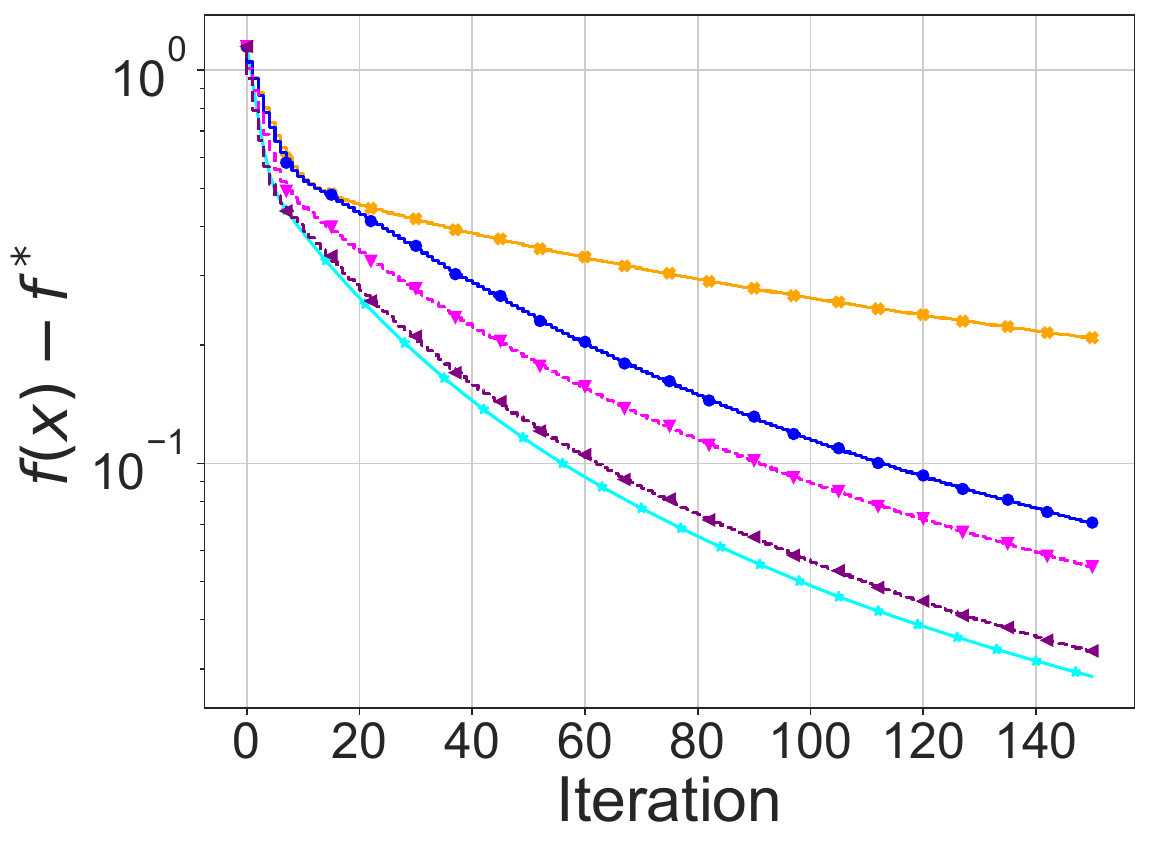}}
  	\subfigure
  	{\includegraphics[width=0.327\textwidth]{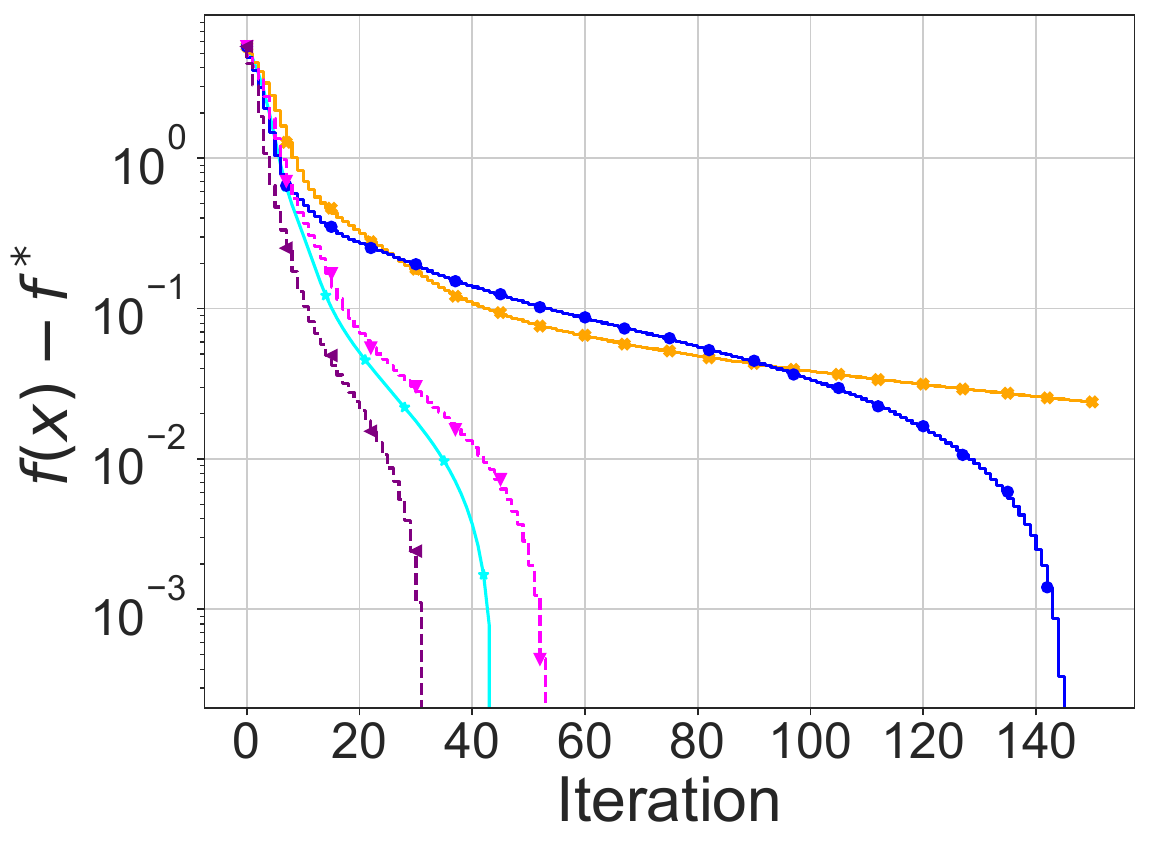}}
  	\subfigure
  	{\includegraphics[width=0.327\textwidth]{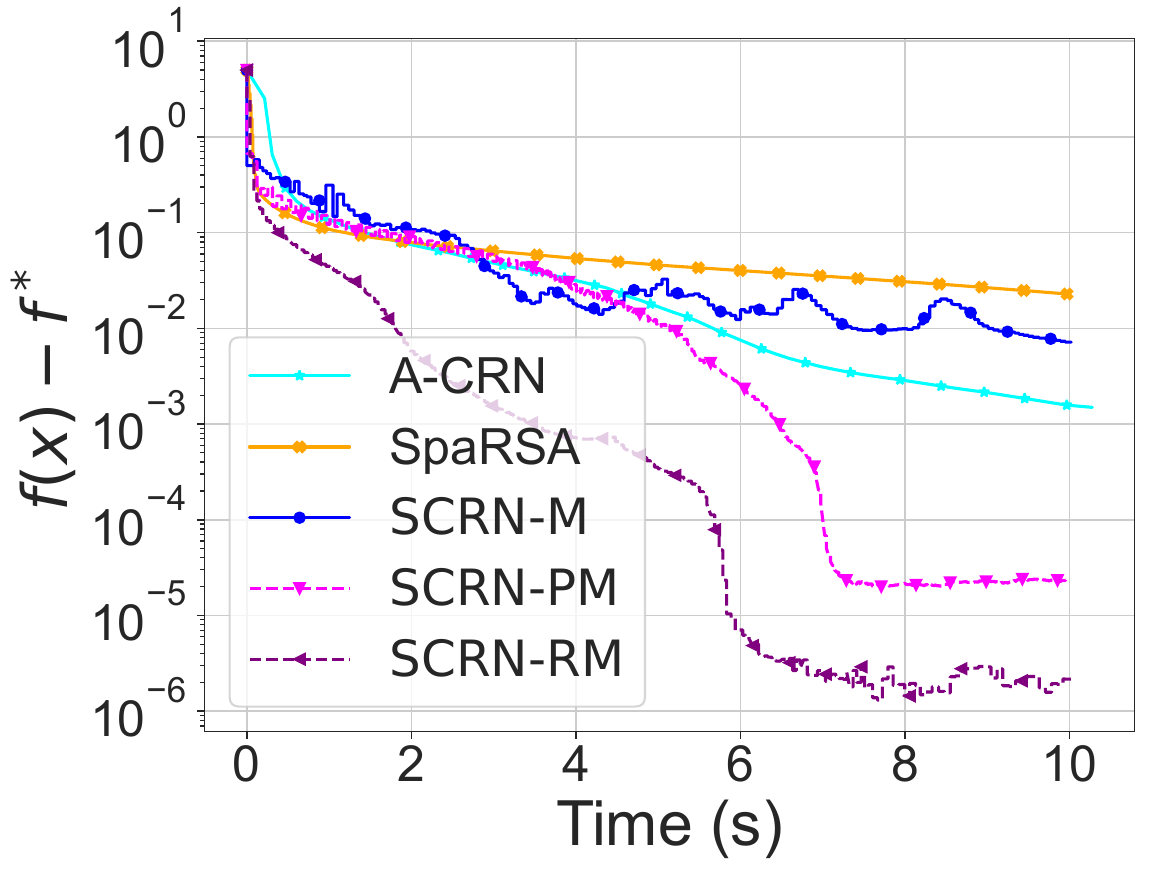}}
  	\subfigure
  	{\includegraphics[width=0.327\textwidth]{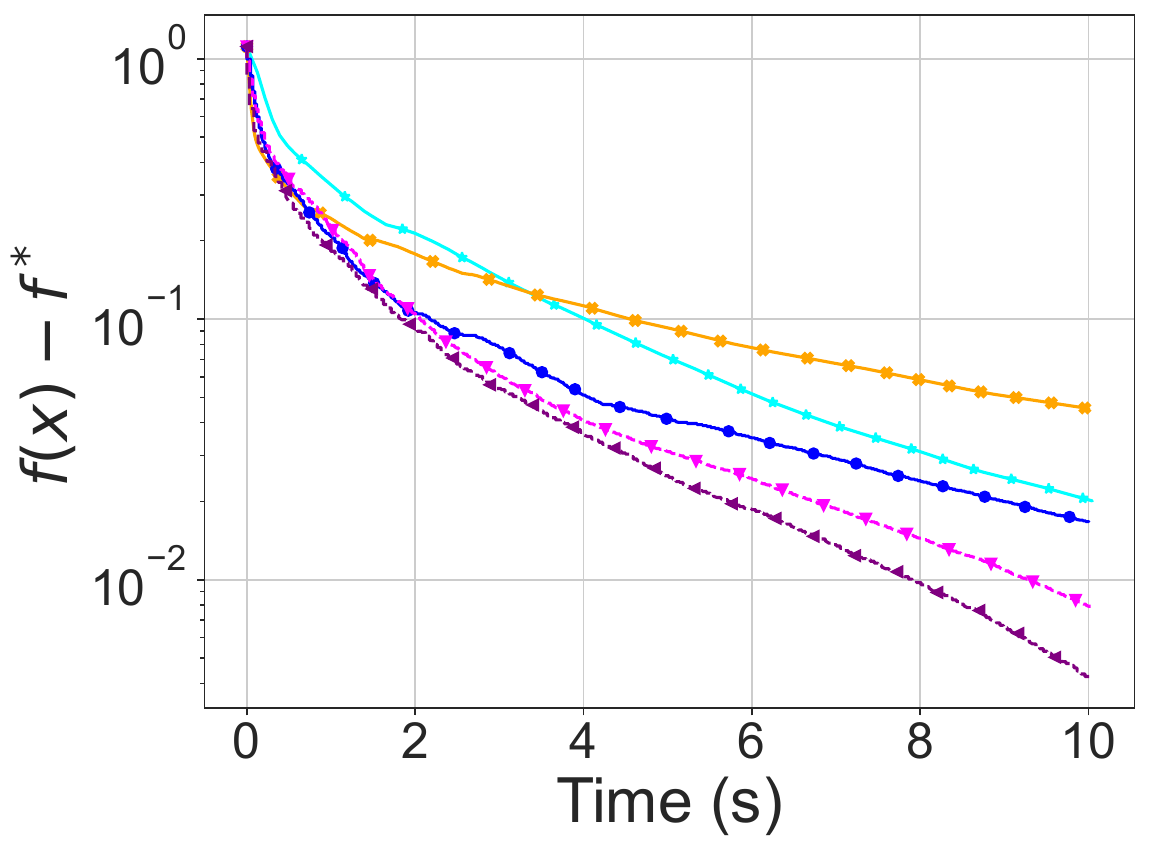}}
  	\subfigure
  	{\includegraphics[width=0.327\textwidth]{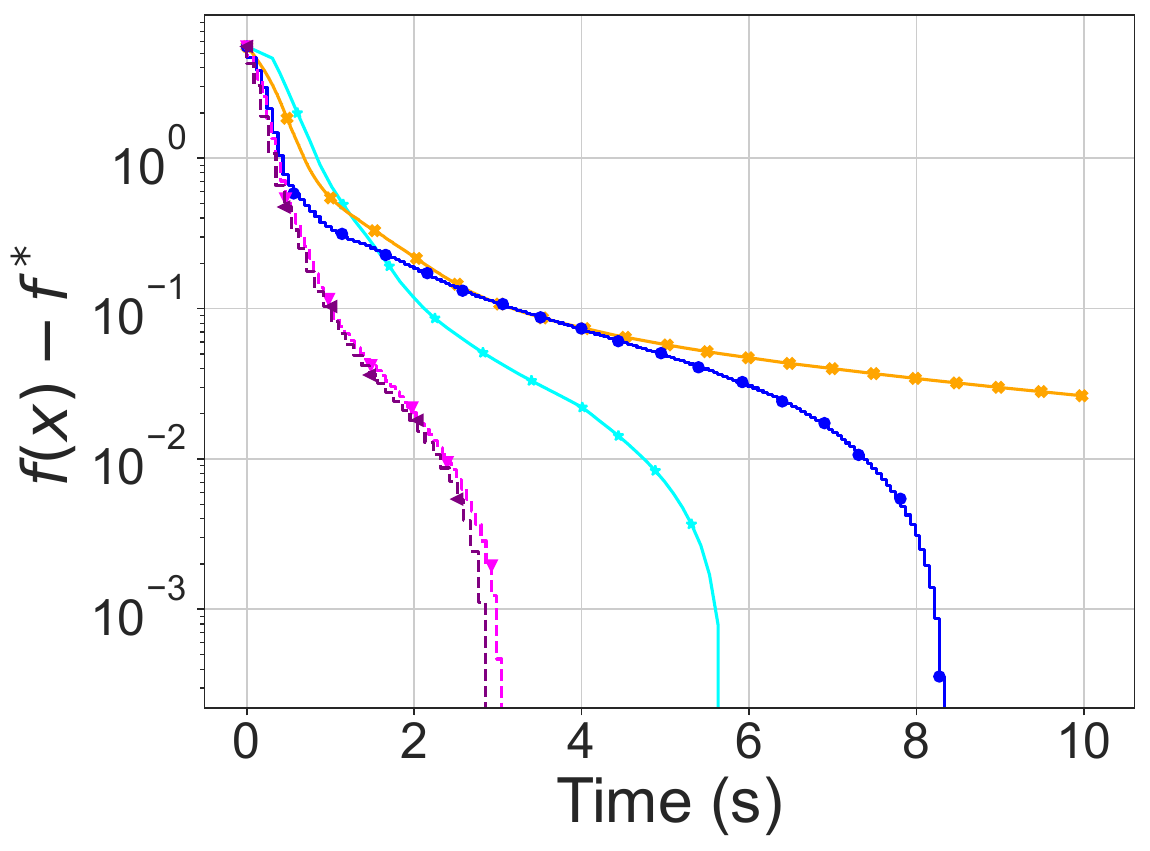}}

  	\caption{Convergence behavior of objective value gap for problem \eqref{logistic-pro}. Correspond to the results on the 'a9a'(left), 'phishing'(middle), and 'w8a'(right) datasets, respectively.}
  	\label{logistic}
  	\end{figure*}

\subsection{Regularized nonlinear least-squares problems}
In this subsection, we consider the regularized nonlinear least-squares problem:
\begin{align}\label{least-square}
\min_{x \in \mathbb{R}^n} \frac{1}{m}\sum_{i=1}^m (b_i - \phi(a_i^{\top} x))^2 + \lambda \sum_{j=1}^n\frac{(\gamma x_j)^2}{1 + (\gamma x_j)^2},  
\end{align}
where $\phi(t) = e^t/(1 + e^t)$ denotes the sigmoid function, $\{(a_{i},b_{i})\}_{1\le i\le m} \subset \R^n\times \R$ is the given data, and $(\lambda,\gamma)=(0.001,1)$. We consider three datasets `a9a', `phishing', and `w8a' from LIBSVM.  

\begin{figure*}[htbp]
  	\centering
  	\subfigure
  	{\includegraphics[width=0.327\textwidth]{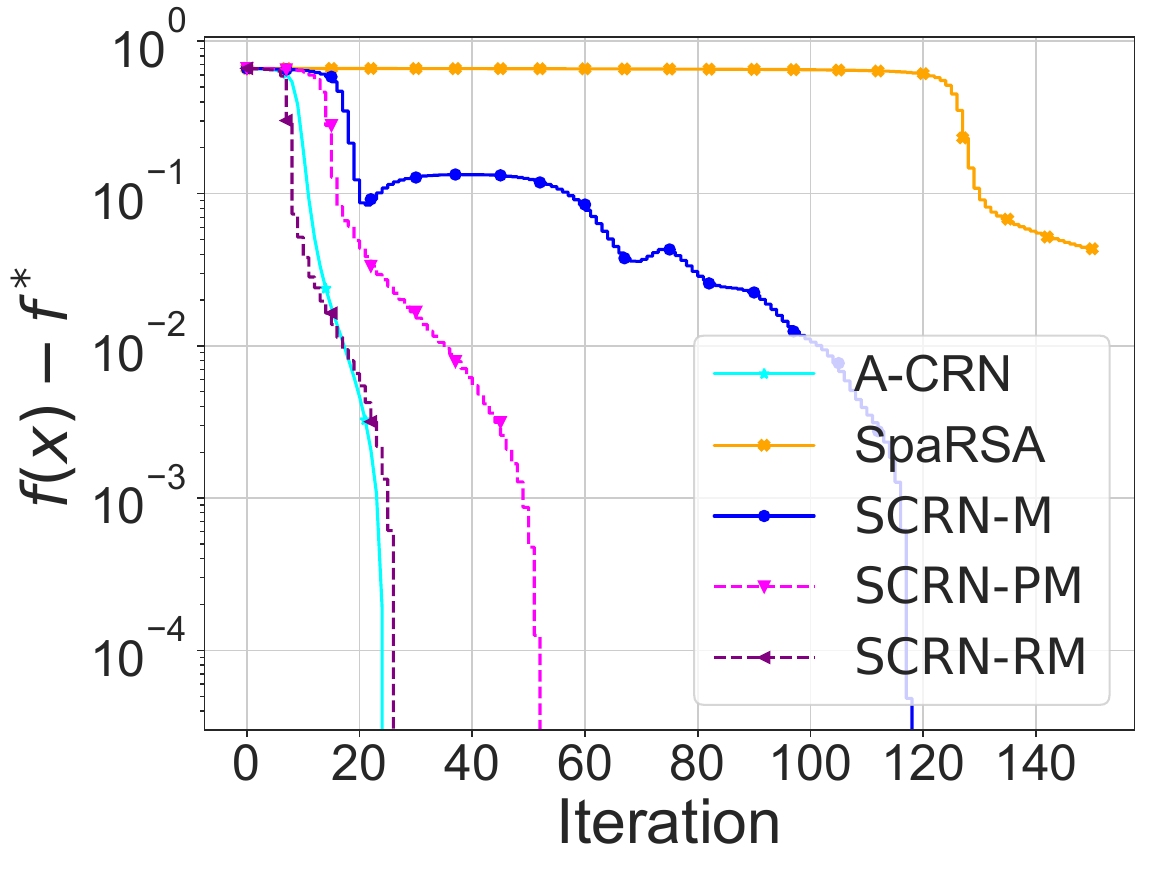}}
  	\subfigure
  	{\includegraphics[width=0.327\textwidth]{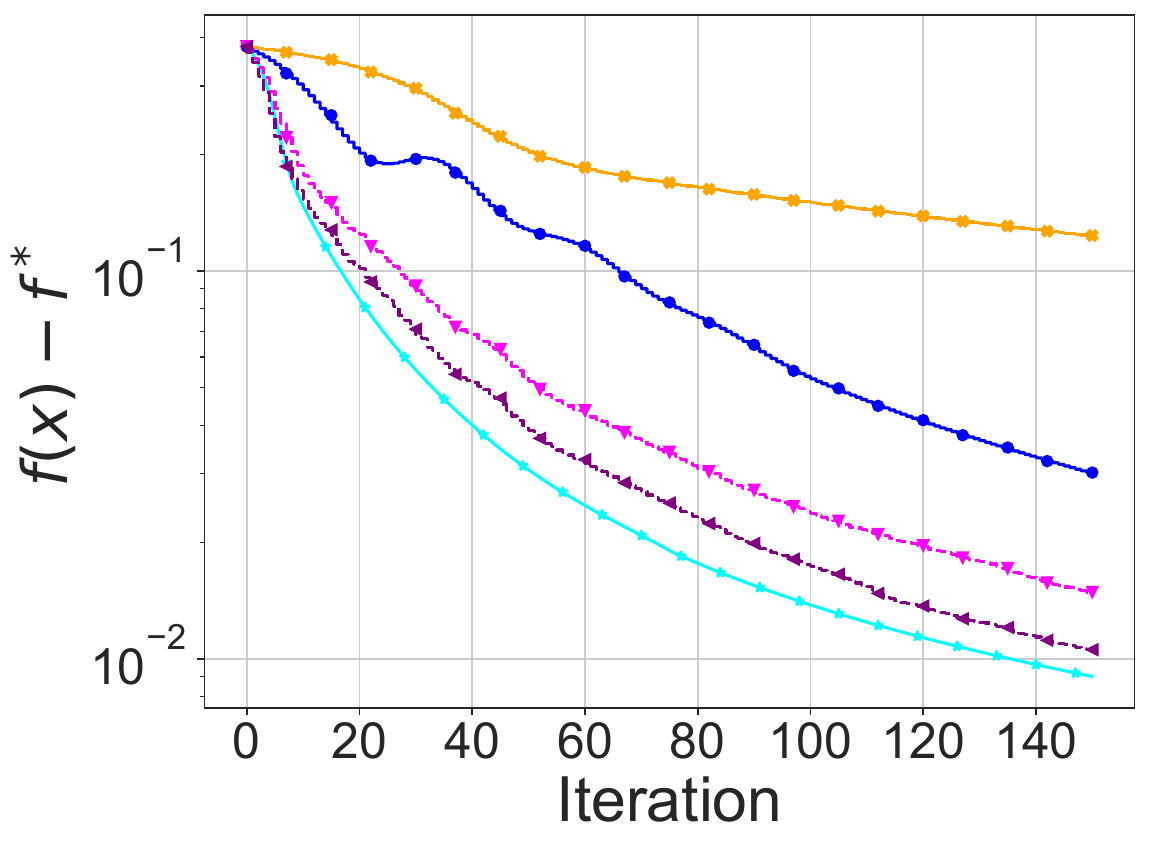}}
  	\subfigure
  	{\includegraphics[width=0.327\textwidth]{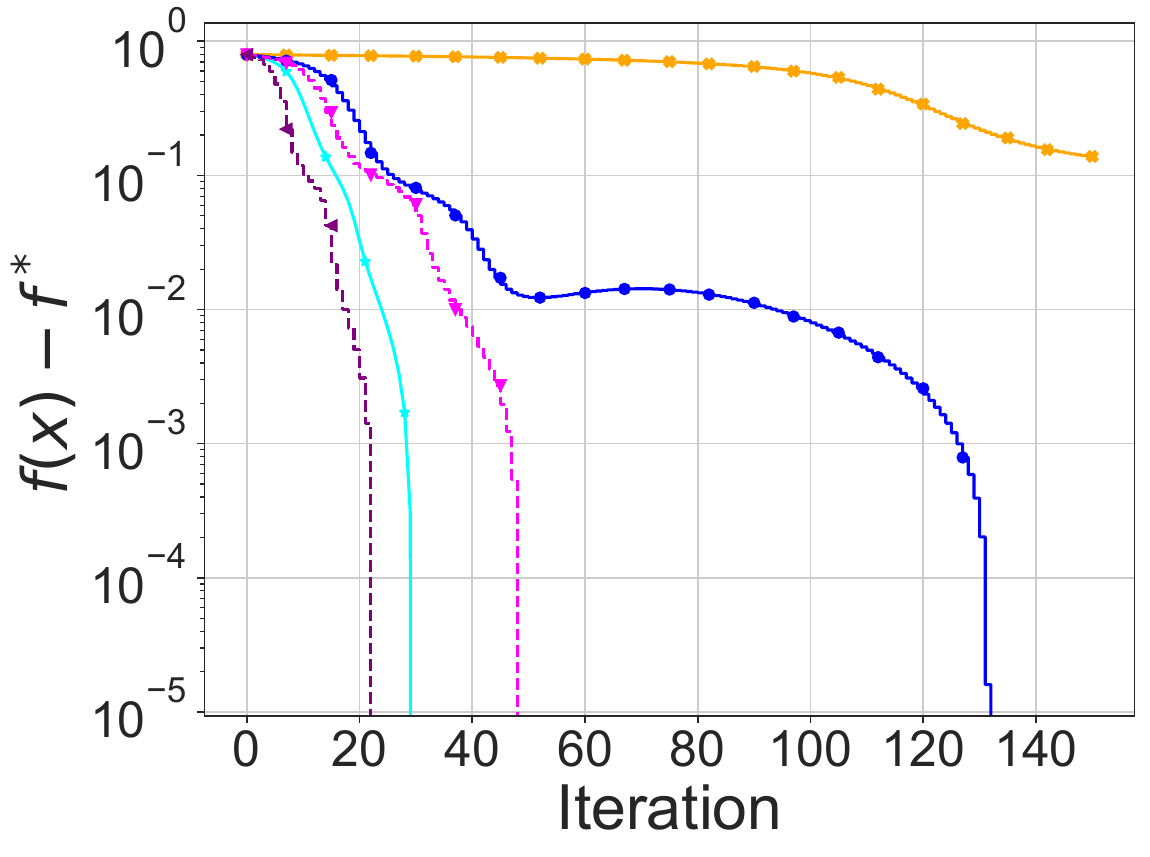}}
  	\subfigure
  	{\includegraphics[width=0.327\textwidth]{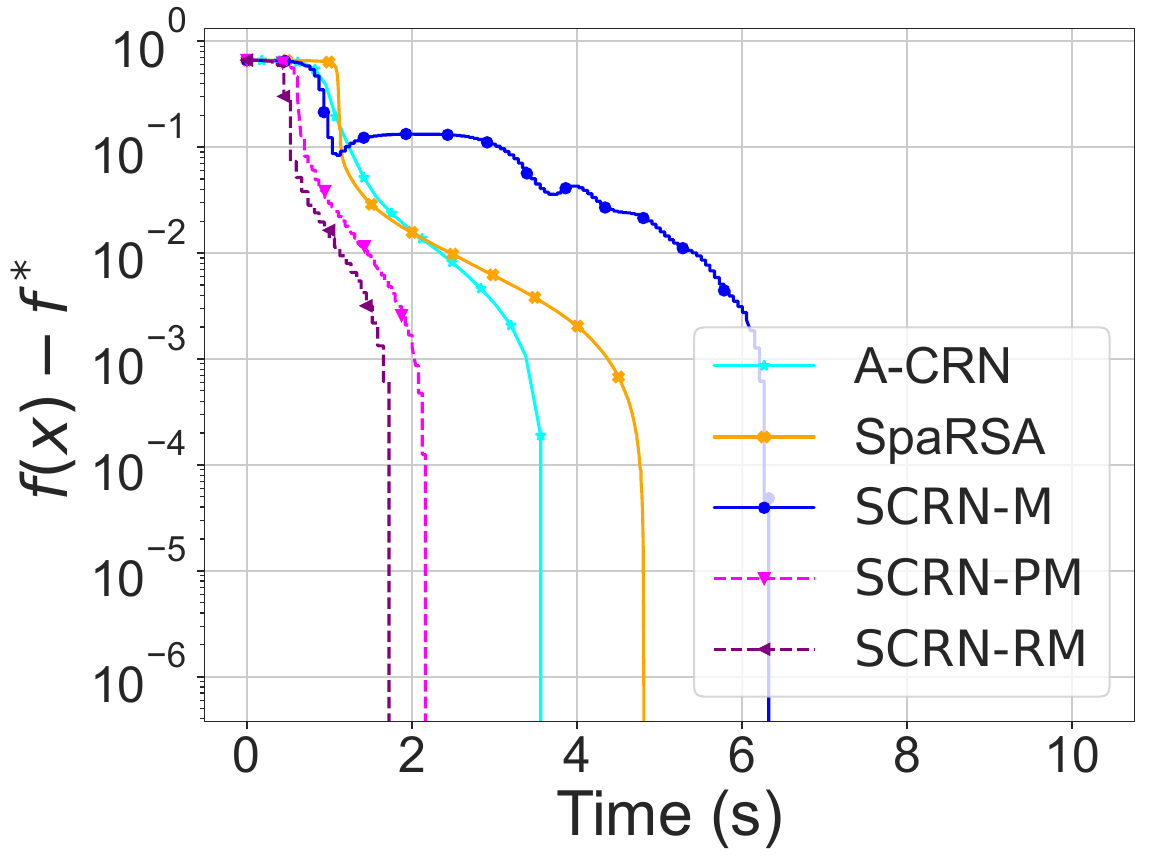}}
  	\subfigure
  	{\includegraphics[width=0.327\textwidth]{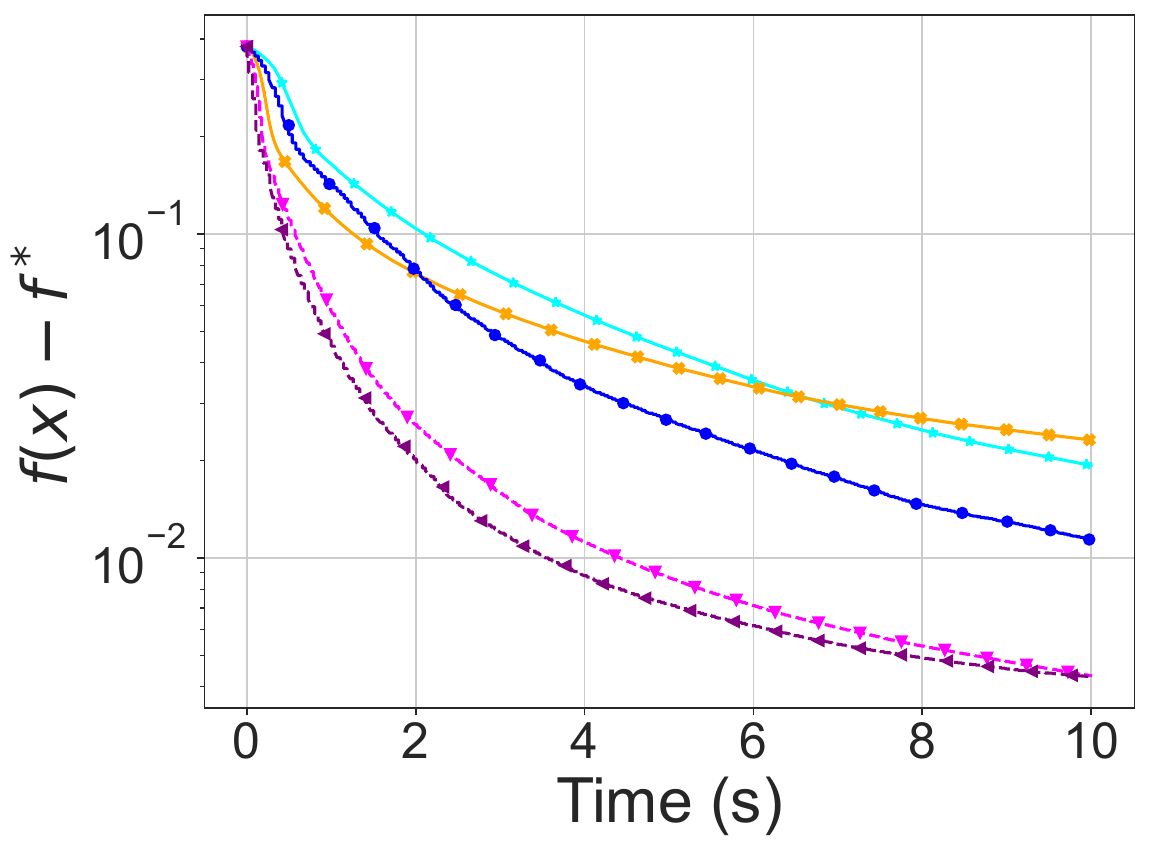}}
  	\subfigure
  	{\includegraphics[width=0.327\textwidth]{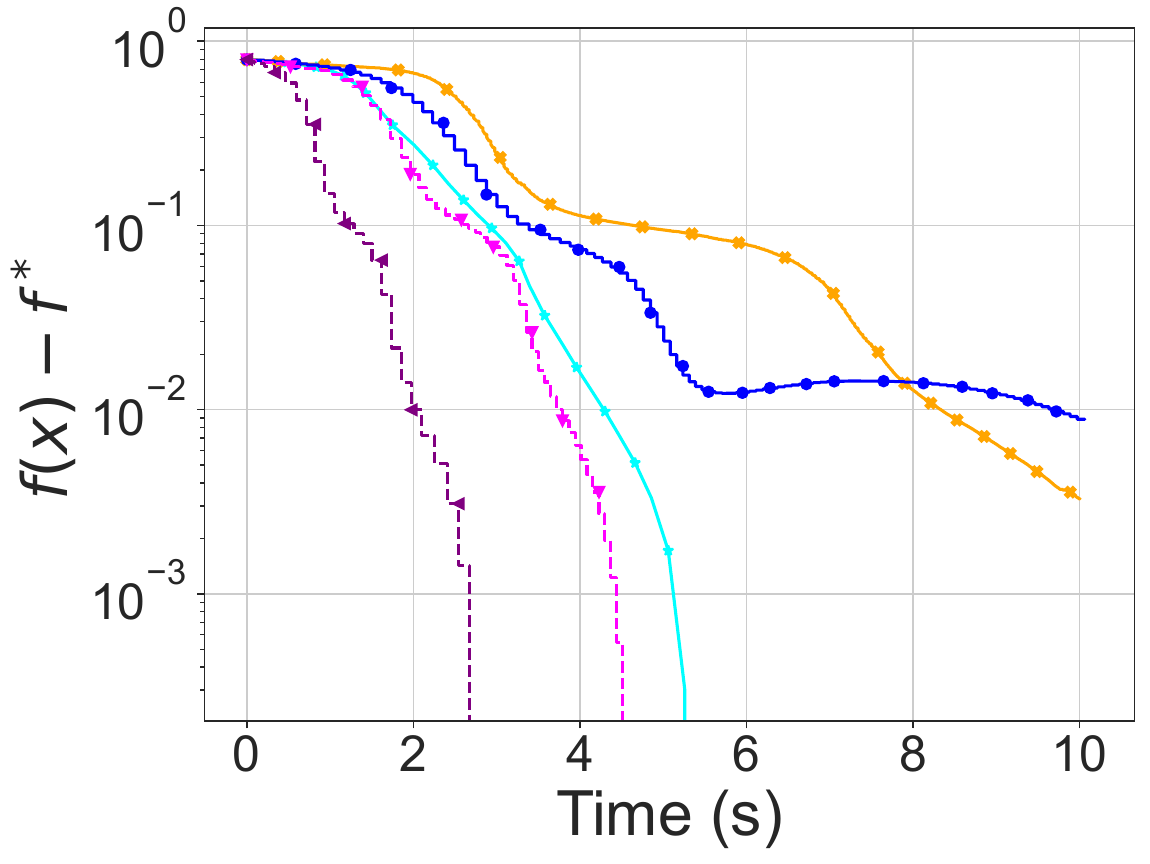}}

  	\caption{Convergence behavior of objective value gap for problem \eqref{least-square}. Correspond to the results on the 'a9a'(left), 'phishing'(middle), and 'w8a'(right) datasets, respectively.}
  	\label{leastsquart}
  	\end{figure*}
We apply SCRN-PM, SCRN-RM, A-CRN, SCRN-M, and SpaRSA to solve problem \eqref{logistic-pro}. All methods are initialized at $[0.5,\ldots,0.5]^T$. For SCRN-M, SCRN-PM, and SCRN-RM, we construct gradient and Hessian estimators using the same strategy as described in Section \ref{subsec:5.1}. For CRN and all SCRN methods, we adopt the Lanczos method used in \cite{cartis2011adaptive} to solve the cubic regularized subproblems. We compare these methods in terms of the function value gap defined by $f(x^k)-f^*$, where $f^*$ is the minimum objective value found during the first $2000$ iterations across all methods. The algorithmic parameters are selected to suit each method well in terms of computational performance.

For each dataset, we plot the function value gap in Figure \ref{leastsquart} to illustrate the convergence behavior of all competing methods. As shown in Figure \ref{leastsquart}, SCRN-PM and SCRN-RM significantly outperform both SCRN-M and SpaRSA. In addition, SCRN-PM and SCRN-RM achieve performance comparable to that of CRN in terms of iteration counts, while outperforming CRN in CPU time. This suggests that using full gradients greatly improves the convergence speed of SCRN, bringing it close to the deterministic CRN while reducing the computational time per iteration. In contrast, when stochastic gradients are used, SCRN converges much more slowly and may offer little to no advantage over first-order methods. In addition, SCRN-RM slightly outperforms SCRN-PM, which is consistent with our theoretical results.

\subsection{Robust linear regression}
In this subsection, we consider the robust linear regression problem:
\begin{align}\label{rob}
 \min_{x \in \mathbb{R}^n} \frac{1}{m} \sum_{i=1}^m \phi(b_i - a_i^{\top} x),   
\end{align}
where $\phi(t) = \ln\left(t^2/2 + 1\right)$ is a nonconvex loss function, $\{(a_{i},b_{i})\}_{1\le i\le n} \subset \R^{m}\times \R$ is the given data, and $(\lambda,\gamma)=(0.001,1)$. We consider three datasets `ijcnn1', `phishing', and `w8a' from LIBSVM.    

We apply SCRN-PM, SCRN-RM, A-CRN, SCRN-M, and SpaRSA to solve problem \eqref{rob}. All methods are initialized at $[0.5,\ldots,0.5]^T$. For SCRN-M, SCRN-PM, and SCRN-RM, we construct gradient and Hessian estimators using the same strategy as described in Section \ref{subsec:5.1}. For CRN and all SCRN methods, we adopt the Lanczos method used in \cite{cartis2011adaptive} to solve the cubic regularized subproblems. We compare these methods in terms of the function value gap defined by $f(x^k)-f^*$, where $f^*$ is the minimum objective value found during the first $2000$ iterations across all methods. The algorithmic parameters are selected to suit each method well in terms of computational performance.

For each dataset, we plot the function value gap in Figure \ref{robust} to demonstrate the convergence behavior of all competing methods. As illustrated, SCRN-PM and SCRN-RM substantially outperform SCRN-M and SpaRSA. In addition, our SCRN-PM and SCRN-RM achieve a comparable performance to CRN in terms of the number of iterations, while outperforming CRN in CPU time. This indicates that incorporating full gradients significantly accelerates the convergence of SCRN, bringing its performance close to that of the deterministic CRN while reducing the computational time per iteration. Conversely, when stochastic gradients are employed, SCRN converges much more slowly and may provide little to no advantage over first-order methods. In addition, SCRN-RM slightly outperforms SCRN-PM, aligning with our theoretical results.



\begin{figure*}[htbp]
  	\centering
  	\subfigure
  	{\includegraphics[width=0.327\textwidth]{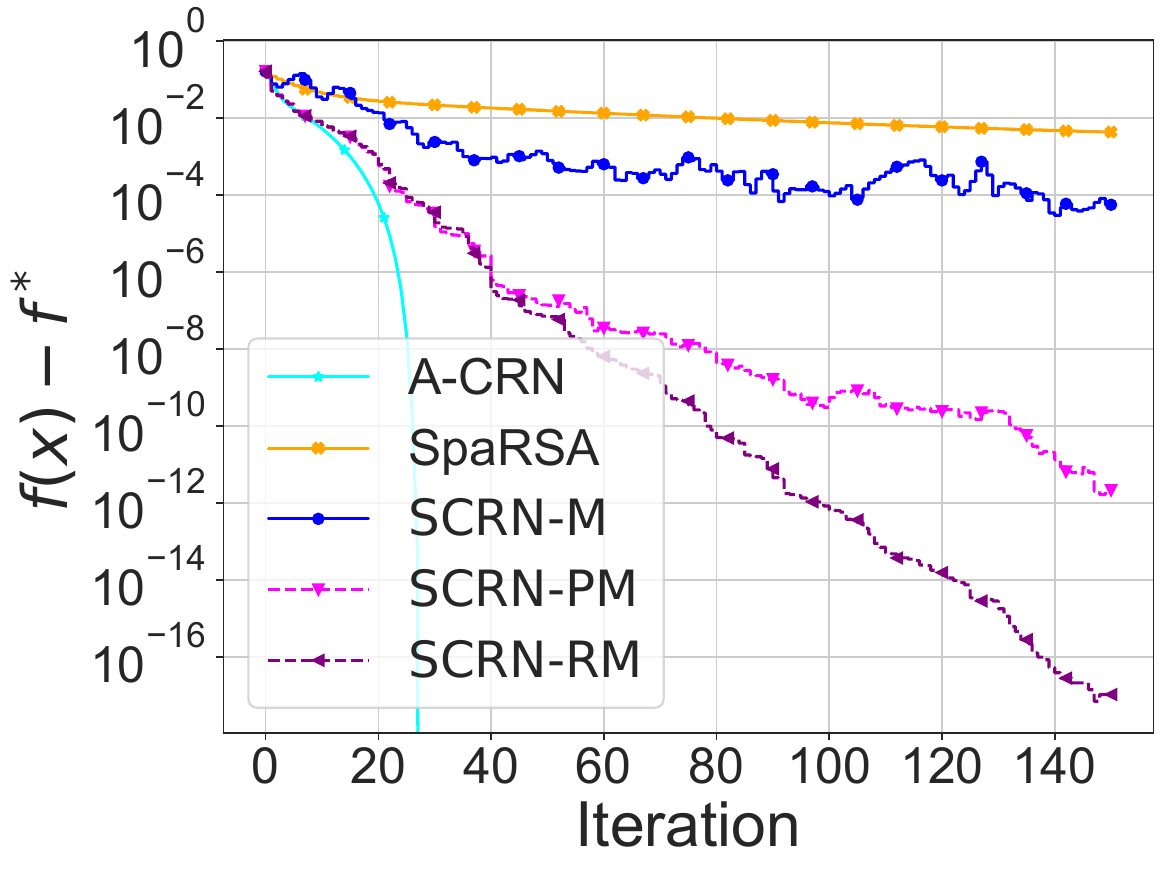}}
  	\subfigure
  	{\includegraphics[width=0.327\textwidth]{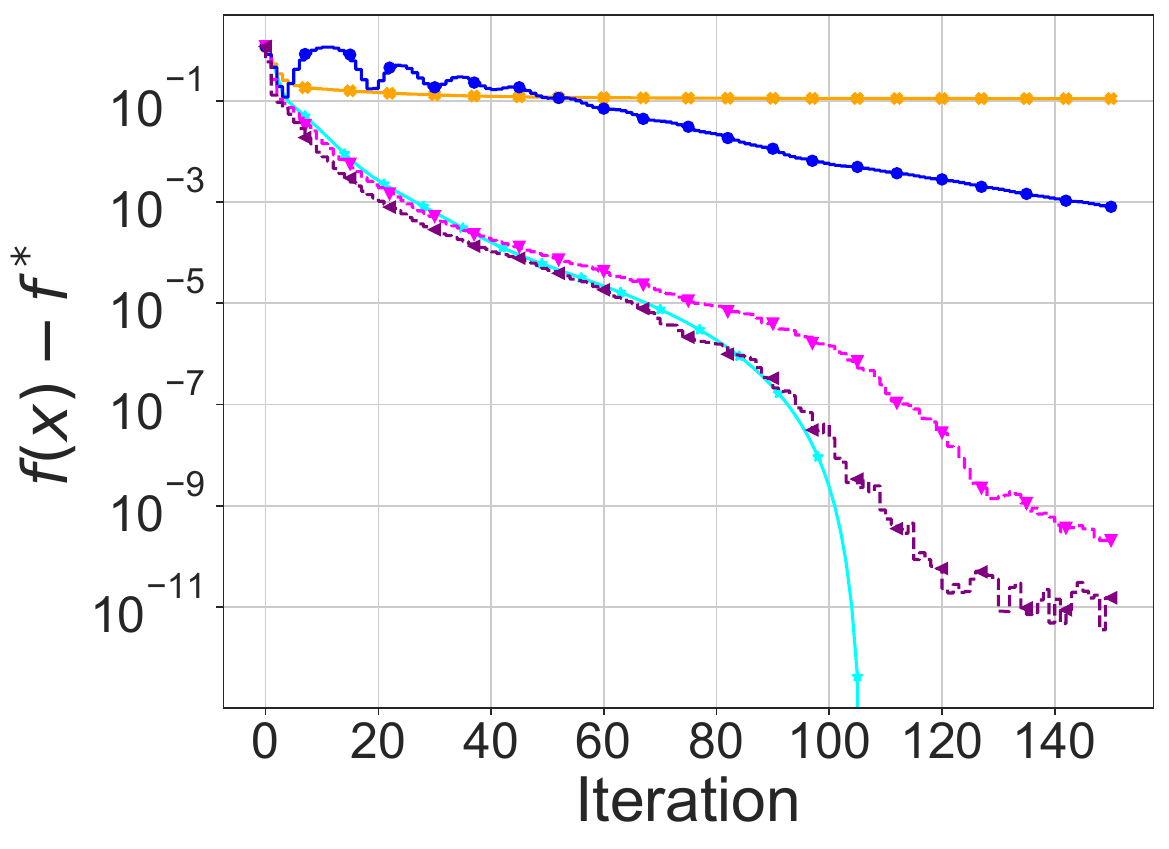}}
  	\subfigure
  	{\includegraphics[width=0.327\textwidth]{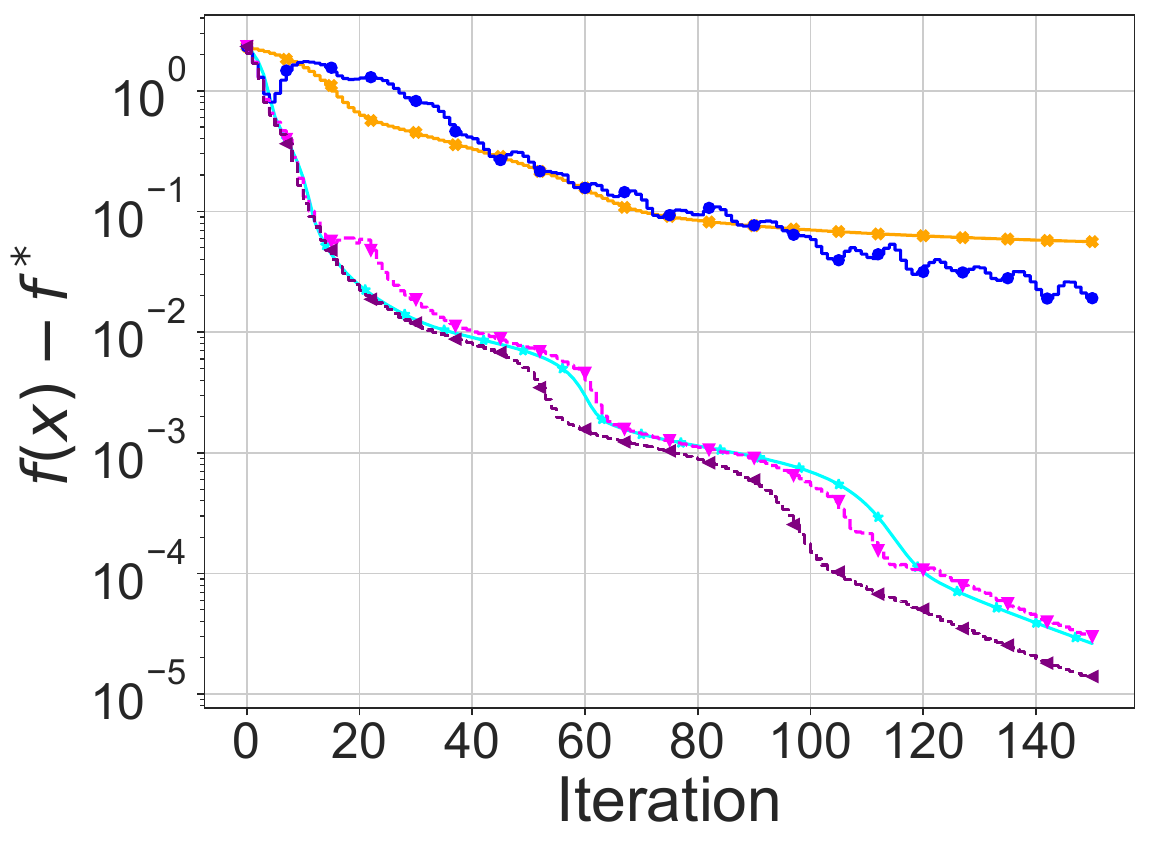}}
  	\subfigure
  	{\includegraphics[width=0.327\textwidth]{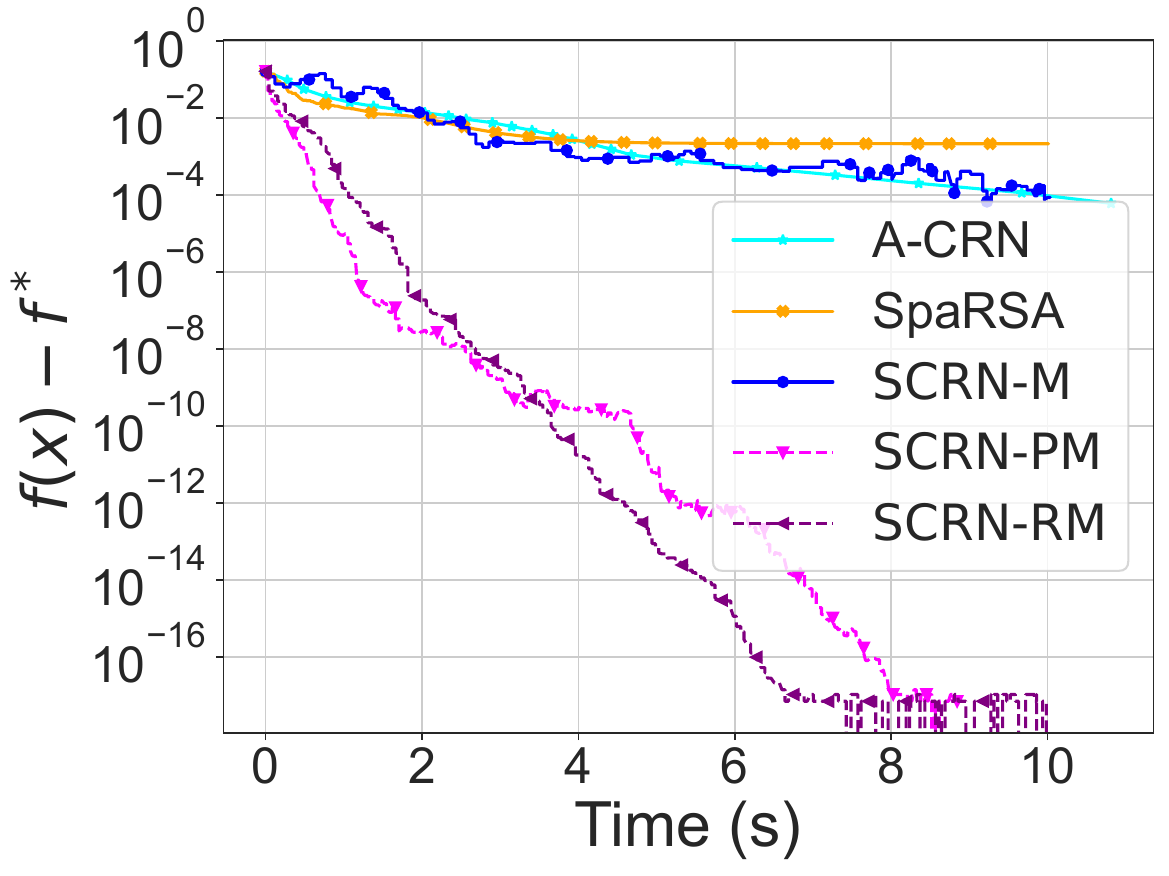}}
  	\subfigure
  	{\includegraphics[width=0.327\textwidth]{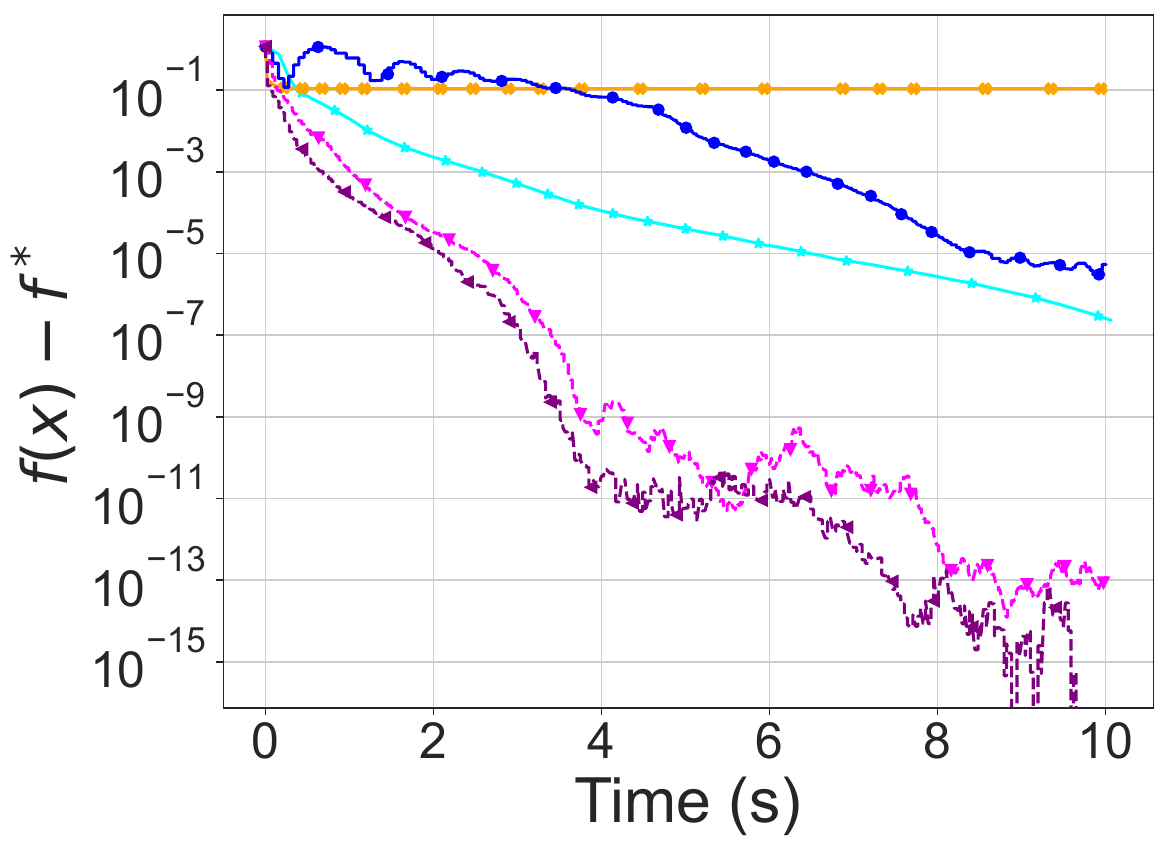}}
  	\subfigure
  	{\includegraphics[width=0.327\textwidth]{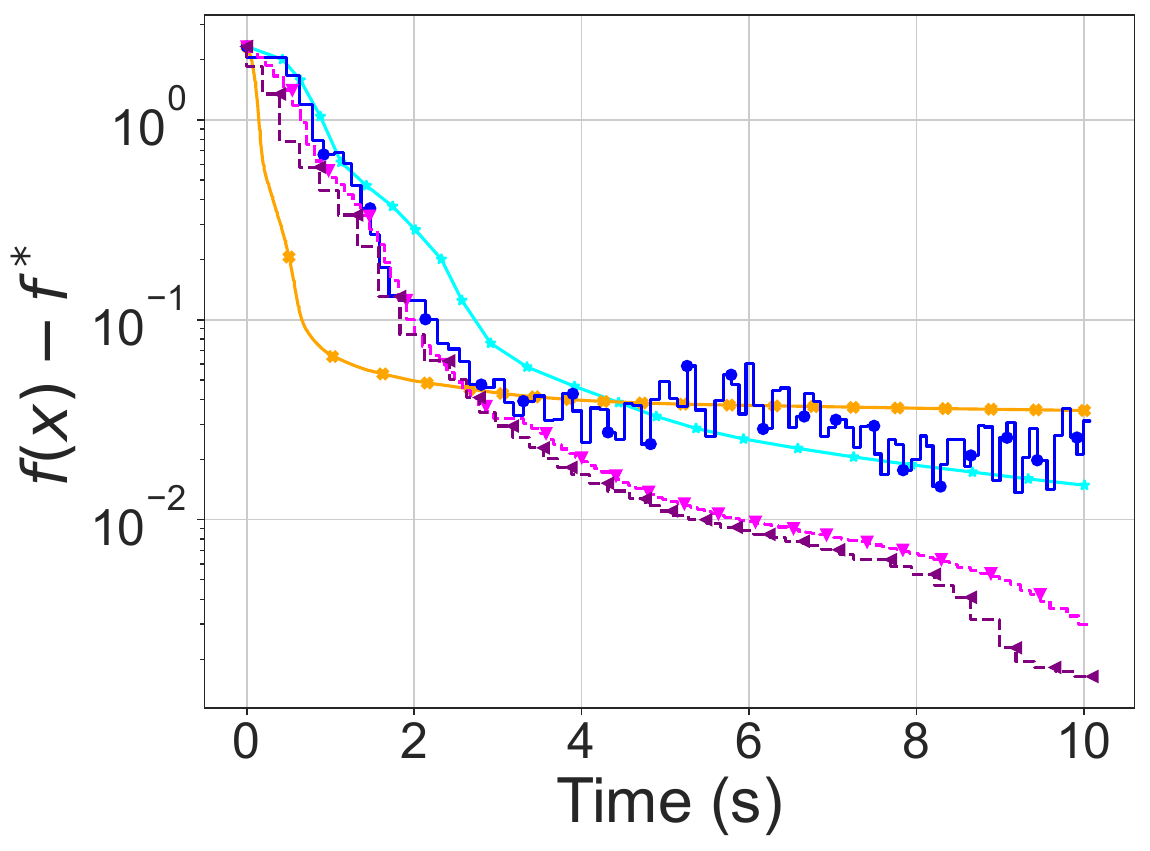}}

  	\caption{Convergence behavior of objective value gap for problem \eqref{rob}. Correspond to the results on the `ijcnn1'(left), `phishing'(middle), and `w8a'(right) datasets, respectively.}
  	\label{robust}
  	\end{figure*}

\section{Proof of the main results}\label{sec:proof}
In this section, we provide proofs of Theorems \ref{th:complexity-pm-c} and \ref{th:complexity-rm-c}. 

For notational convenience, we define a sequence of potentials for Algorithms \ref{alg:unf-ssom-pm} and \ref{alg:unf-ssom-rm} as 
\begin{align}\label{def:pot}
\mathcal{P}_k:= f(x^k) + p_k\|M_k - \nabla^2 f(x^k)\|_F^3 \qquad\forall k\ge 0,    
\end{align}
where the sequence $\{(x^k,M_k)\}$ is generated by each respective algorithm, and $\{p_k\}$ is a sequence of positive scalars that will be specified separately for each case. We also define the following quantity for measuring the approximate first- and second-order stationarity of problem \eqref{ucpb}:
\begin{align}
\mu_\eta(x) := \max\Big\{\frac{1}{3}\|\nabla f(x)\|^{3/2}, -\frac{\eta^{3/2}}{4}\lambda_{\min}(\nabla^2 f(x))^3\Big\} \label{eq:mu_eta} 
\end{align}
for some $\eta>0$. 

The following lemma provides expansions for the cubed Frobenius norm, generalizing the well-known identity $\|U + V\|_F^2 = \|U\|_F^2 + 2\mathrm{Tr}(U^TV) + \|V\|_F^2$ and inequality $\|U + V\|_F^2\le (1+c)\|U\|_F^2+(1+1/c)\|V\|_F^2$ for all $U,V\in\mathbb{R}^{n\times n}$ and $c>0$. 

\begin{lemma}\label{lem:expand-nonsquare1}
For any $U,V\in\R^{n\times n}$, it holds that  
\begin{align}
&\|U+V\|_F^3 \le (1+c)\|U\|_F^3 + 3\|U\|_F\mathrm{Tr}(U^TV) + 2(1 + c^{-1/2})\|V\|_F^3\quad\forall c>0,\label{cubic-function-1}\\
&\|U+V\|_F^3 \le (1+2c)\|U\|_F^3 + 2(1 + c^{-1/2} + 2c^{-2})\|V\|_F^3\quad \forall c>0.\label{cubic-function-2}
\end{align}
\end{lemma}

\begin{proof}
Fix any $U,V\in\R^{n\times n}$. For convenience, we vectorize $U$ and $V$ by letting $u=\mathrm{vec}(U)\in\R^{n^2}$ and $v=\mathrm{vec}(V)\in\R^{n^2}$. Let $\phi(w):=\|w\|^3$ for all $w\in\R^{n^2}$. It follows from \cite[Theorem 6.3]{rodomanov2020smoothness} that 
\begin{align*}
\|\nabla^2\phi(w) - \nabla^2\phi(w^\prime)\| \le 9 \|w-w^\prime\|\quad\forall w,w^\prime\in\R^{n^2}.   
\end{align*}
By this and \eqref{ineq:2st-desc}, one has that  
\begin{align*}
\phi(u+v)\le \phi(u) + \nabla\phi(u)^Tv + \frac{1}{2}v^T \nabla^2\phi(u)v + \frac{3}{2}\|v\|^3.
\end{align*}
This together with $\phi(u)=\|u\|^3$, $\nabla\phi(u)= 3\|u\|u$, and $\nabla^2 \phi(u)=3(uu^T/\|u\| + \|u\|I)$ implies that
\begin{align}
&\|u+v\|^3 \le \|u\|^3 + 3\|u\|u^Tv + 3v^T (uu^T/\|u\| + \|u\|I)v/2 + 2\|v\|^3\nonumber\\
&\le \|u\|^3 + 3\|u\|u^Tv + 3 \|u\| \|v\|^2 + 2\|v\|^3 \le (1+c)\|u\|^3 + 3\|u\|u^Tv + 2(1 + c^{-1/2})\|v\|^3\quad \forall c>0,\label{rel-1}
\end{align} 
where the last inequality is due to the Young's inequality. Using again the Young's inequality and \eqref{rel-1}, we obtain that 
\begin{align}
\|u+v\|^3\le (1+2c)\|u\|^3 + 2(1 + c^{-1/2} + 2c^{-2})\|v\|^3\quad\forall c>0.\label{rel-2}
\end{align}
In view of \eqref{rel-1}, \eqref{rel-2}, $u=\mathrm{vec}(U)$, and $v=\mathrm{vec}(V)$, we see that \eqref{cubic-function-1} and \eqref{cubic-function-2} hold as desired.
\end{proof}

\subsection{Some properties of cubic subproblems}

In this subsection, we present some properties of the cubic regularized subproblem:
\begin{align}\label{eq:sub-cubic}
x^+ \in \underset{x^\prime \in \mathbb{R}^{n}}{\mathrm{Arg\,min}} \Big\{ g^T (x^\prime-x) + \frac{1}{2}(x^\prime-x)^TM(x^\prime-x) + \frac{1}{6\eta}\|x^\prime-x\|^3\Big\}   
\end{align}
for given $x\in\R^n$, $M\in\R^{n\times n}$, and $\eta>0$. The first- and second-order optimality condition of \eqref{eq:sub-cubic} yield
\begin{align}
g + M(x^+-x) + \frac{1}{2\eta}\|x^+-x\|(x^+-x) = 0,\quad M + \frac{1}{2\eta}\|x^+ - x\|I \succeq 0.\label{eq:second-cond}
\end{align}

The next lemma provides an upper bound for the first- and second-order stationary measure at $x^+$, which can be seen as an inexact variant of \cite[Lemma 5]{nesterov2006cubic}.

\begin{lemma}\label{lem:bound-fsp}
Suppose that Assumption \ref{asp:basic} holds. Assume that $\eta\in(0,(2L)^{-1})$ holds, where $L$ is given in Assumption \ref{asp:basic}(b). Let $x\in\R^n$ and $M\in\R^{n\times n}$ be given, and let $x^+$ be a solution to \eqref{eq:sub-cubic}. Then,
\begin{align}
\|\nabla f(x^+)\|^{3/2} & \le \frac{3}{\eta^{3/2}}\|x^+-x\|^3 + \frac{3\eta^{3/2}}{4}\|M-\nabla^2 f(x)\|_F^3 + 3\|g-\nabla f(x)\|^{3/2},\label{ineq:grad}\\
- \eta^{3/2}\lambda_{\min}(\nabla^2f(x^+))^3 & \le \frac{4}{\eta^{3/2}}\|x^+-x\|^3 + 4\eta^{3/2}\|\nabla^2f(x)-M\|_F^3.\label{ineq:lambda}
\end{align}
Consequently, one has 
\begin{align}
\mu_\eta(x^+) \le {\eta^{-3/2}}\|x^+ - x\|^3 + \eta^{3/2}\|M - \nabla^2 f(x)\|_F^3 + \|g-\nabla f(x)\|^{3/2},\label{ineq:mu}  
\end{align}
where $\mu_\eta$ is defined in \eqref{eq:mu_eta}.
\end{lemma}
\begin{proof}
Using \eqref{ineq:1st-desc} with $y=x^+$, we obtain that 
\begin{align*}
\| \nabla f(x^+) - g - M(x^{+} - x) + g - \nabla f(x) + (M -\nabla^2 f(x))(x^+-x) \|  \le \frac{L}{2}\|x^+-x\|^2.
\end{align*}
This along with the first relation in \eqref{eq:second-cond} implies that
\begin{align*}
\|\nabla f(x^+)\| & \le \frac{L}{2}\|x^+-x\|^2 + \|g + M(x^{+}-x)\| + \|g-\nabla f(x)\| + \|(M-\nabla^2f(x))(x^+-x)\|\\ 
&\overset{\eqref{eq:second-cond}}{=} \Big(\frac{L}{2} + \frac{1}{2\eta}\Big)\|x^+-x\|^2 + \|g-\nabla f(x)\| + \|(M-\nabla^2f(x))(x^+-x)\|\\
&\le \frac{3}{4\eta}\|x^+-x\|^2 + \|g-\nabla f(x)\| + \|M-\nabla^2 f(x)\|\|x^+-x\| \\
& \le \frac{5}{4\eta}\|x^+-x\|^2 + \frac{\eta}{2}\|M-\nabla^2 f(x)\|^2 + \|g-\nabla f(x)\|
\end{align*}
where the second inequality is due to the spectral norm inequality and $L\le 1/(2\eta)$, and the third inequality follows from the Young's inequality. This inequality further implies that 
\begin{align*}
\|\nabla f(x^+)\|^{3/2} & \le \Big(\frac{5}{4\eta}\|x^+-x\|^2 + \frac{\eta}{2}\|M-\nabla^2 f(x)\|^2 + \|g-\nabla f(x)\|\Big)^{3/2}\\
& \le \sqrt{3}\Big(\Big(\frac{5}{4\eta}\Big)^{3/2}\|x^+-x\|^3 + \Big(\frac{\eta}{2}\Big)^{3/2}\|M-\nabla^2 f(x)\|^3 + \|g-\nabla f(x)\|^{3/2}\Big)\\
& \le \frac{3}{\eta^{3/2}}\|x^+-x\|^3 + \frac{3\eta^{3/2}}{4}\|M-\nabla^2 f(x)\|^3 + 3\|g-\nabla f(x)\|^{3/2},
\end{align*}
where the second inequality is due to $(a+b+c)^{3/2}\le \sqrt{3}(a^{3/2}+b^{3/2}+c^{3/2})$ for all $a,b,c\ge0$. This together with the fact that the spectral norm of a matrix is bounded above by the Frobenius norm proves \eqref{ineq:grad} as desired.

We next prove \eqref{ineq:lambda}. Using the Lipschitz continuity of $\nabla^2 f$ and \eqref{eq:second-cond}, we obtain that
\begin{align*}
\nabla^2 f(x^+) & \succeq \nabla^2 f(x) - L\|x^+-x\| I \succeq M - \|M - \nabla^2 f(x)\|I - L\|x^+-x\| I \\ 
& \overset{\eqref{eq:second-cond}}{\succeq} -((2\eta)^{-1}+L)\|x^+-x\|I - \|M - \nabla^2 f(x)\|I \succeq \eta^{-1}\|x^+-x\|I - \|M - \nabla^2 f(x)\|I,
\end{align*}
where the last relation is due to $L\le (2\eta)^{-1}$. It then follows that 
\begin{align*}
-\lambda_{\min}(\nabla^2 f(x^+))^3 \le  (\eta^{-1} \|x^+-x\| + \|M - \nabla^2 f(x)\|)^3 \le 4\eta^{-3} \|x^+-x\|^3 + 4\|M - \nabla^2 f(x)\|^3,
\end{align*}
where the second inequality is due to $(a+b)^{3}\le 4a^3 + 4b^3$ for all $a,b\ge0$. In view of the above inequality and the fact that the spectral norm of a matrix is bounded above by the Frobenius norm, we obtain that \eqref{ineq:lambda} holds. 

Combining \eqref{ineq:grad} and \eqref{ineq:lambda} with the definition of $\mu_\eta$ in \eqref{eq:mu_eta}, we obtain that \eqref{ineq:mu} holds, which completes the proof of this lemma.
\end{proof}

We next show that solving a cubic regularized subproblem yields a descent property of $f$.

\begin{lemma}\label{lem:ppt-desc}
Suppose that Assumption \ref{asp:basic} holds. Assume that $\eta\in(0,(2L)^{-1})$ holds, where $L$ is given in Assumption \ref{asp:basic}(b). Let $x\in\R^n$ and $M\in\R^{n\times n}$ be given, and let $x^+$ be a solution to \eqref{eq:sub-cubic}. Then,
\begin{align}\label{lem:descent}
f(x^+) \le f(x) - \frac{1}{9\eta}\|x^+-x\|^3 + 24\eta^2\|\nabla^2 f(x)-M\|_{F}^3 + 3\eta^{1/2} \|\nabla f(x)-g\|^{3/2}.
\end{align}
\end{lemma}

\begin{proof} 
It follows from \eqref{eq:second-cond} that 
\begin{align}
g^T(x^+-x) & = - (x^+-x)^T M(x^+-x) - \frac{1}{2\eta}\|x^+-x\|^3,\label{ineq:combine-condition}\\
-(x^+-x)^TM(x^+-x) & \le \frac{1}{2\eta}\|x^+-x\|^3.\label{ineq:combine-condition-1}
\end{align}
Using these and \eqref{ineq:2st-desc} with $y=x^+$, we obtain that
\begin{align*}
&f(x^+) \overset{\eqref{ineq:2st-desc}}{\le} f(x) + \nabla f(x)^T(x^{+}-x) + \frac{1}{2}(x^+-x)^T\nabla^2 f(x)(x^+-x)+ \frac{L}{6}\|x^+-x\|^3\\
& = f(x) + g^T(x^+-x) + \frac{1}{2}(x^+-x)^T M(x^+-x) + \frac{L}{6}\|x^+-x\|^3\\
&\qquad + (\nabla f(x)-g)^{T}(x^+-x) + \frac{1}{2}(x^+-x)^T(\nabla^2 f(x) - M )(x^+-x)\\
& \overset{\eqref{ineq:combine-condition}}{=} f(x) - \frac{1}{2}(x^+-x)^T M(x^+-x) - \Big(\frac{1}{2\eta}- \frac{L}{6}\Big)\|x^+-x\|^3\\
&\qquad + (\nabla f(x)-g)^{T}(x^+-x) + \frac{1}{2}(x^+-x)^T(\nabla^2 f(x) - M)(x^+-x)\\
& \overset{\eqref{ineq:combine-condition-1}}{\le} f(x) - \Big(\frac{1}{4\eta}-\frac{L}{6}\Big)\|x^+-x\|^3 + (\nabla f(x)-g)^{T}(x^+-x)  + \frac{1}{2}(x^+-x)^T(\nabla^2 f(x) - M)(x^+-x)\\
& \le f(x) - \frac{1}{6\eta}\|x^+-x\|^3 + \|x^+-x\|\|\nabla f(x)-g\| + \frac{1}{2}\|x^+-x\|^2\|\nabla^2 f(x) - M\|\\
&\le f(x) - \frac{1}{9\eta}\|x^+-x\|^3 + 24\eta^2\|\nabla^2 f(x)-M\|^3 + 3\eta^{1/2} \|\nabla f(x)-g\|^{3/2} ,
\end{align*}
where the third inequality is due to $L \le (2\eta)^{-1}$ and the spectral norm inequality, and the last inequality is due to Young's inequality in two forms: $ab\le a^{3}/(36\eta) + 2\sqrt{12}\eta^{1/2} b^{3/2}/3$ and $ab\le a^{3/2}/(18\eta) + 48\eta^2 b^3$ for all $a,b>0$. In view of the above inequality and the fact that the spectral norm of a matrix is bounded above by the Frobenius norm, we obtain that this lemma holds as desired. 
\end{proof}

\subsection{Proof of the main results in Section \ref{sec:scn-pm}}\label{subsec:proof-pm}
In this subsection, we present some technical lemmas and then use them to prove Theorem \ref{th:complexity-pm-c}. The following lemma gives the recurrence for the estimation error of the Hessian estimators $\{M_k\}$ generated by Algorithm \ref{alg:unf-ssom-pm}. 
\begin{lemma}\label{lem:rec-Hes-pm}
Suppose that Assumption \ref{asp:basic} holds. Let $\{(x^k,M_k)\}$ be the sequence generated by Algorithm \ref{alg:unf-ssom-pm} with momentum parameters $\{\theta_k\}$. Then, it holds that for all $k\ge0$,
\begin{align}
\E_{\xi^{k+1}}[\|M_{k+1} - \nabla^2 f(x^{k+1})\|_F^3]\le (1-\theta_k)\|M_k-\nabla^2 f(x^k)\|_F^3 + 21L_{F}^3\theta_k^{-2}\|x^{k+1}-x^k\|^3 + 5\sigma^3\theta_k^{5/2},\label{hessian-vr}
\end{align}
where $L_F$ and $\sigma$ are given in Assumption \ref{asp:basic}.
\end{lemma}

\begin{proof}
Fix any $k\ge0$. It follows from \eqref{update-Mk-pm} that 
\begin{align}
&M_{k+1}-\nabla^2f(x^{k+1})  \overset{\eqref{update-Mk-pm}}{=} (1 - \theta_{k}) M_{k} + \theta_{k} H(x^{k+1};\xi^{k+1}) - \nabla^2 f(x^{k+1})\nonumber\\
& = (1-\theta_k)(M_k-\nabla^2 f(x^k)) + (1-\theta_k)(\nabla^2 f(x^k) - \nabla^2 f(x^{k+1})) + \theta_k(H(x^{k+1};\xi^{k+1}) - \nabla^2 f(x^{k+1})).\label{a-identity-pm}
\end{align}
Observe from Assumption \ref{asp:basic} that $\|\nabla^2 f(x^{k+1}) - \nabla^2 f(x^k)\|_F\le L_F\|x^{k+1}-x^k\|$, $\E_{\xi^{k+1}}[H(x^{k+1};\xi^{k+1})]=\nabla^2 f(x^{k+1})$ and $\E_{\xi^{k+1}}[\|H(x^{k+1};\xi^{k+1}) -\nabla^2 f(x^{k+1})\|^3_F]\le\sigma^3$. Using these, \eqref{cubic-function-1}, \eqref{cubic-function-2}, and \eqref{a-identity-pm}, we obtain that for all $c>0$,
\begin{align}
&\E_{\xi^{k+1}}[\|M_{k+1} - \nabla^2 f(x^{k+1})\|_F^3] 
\nonumber\\ 
&\overset{\eqref{a-identity-pm}}{=}\E_{\xi^{k+1}}\big[\|(1-\theta_k)(M_k-\nabla^2 f(x^k)) + (1-\theta_k)(\nabla^2 f(x^k) - \nabla^2 f(x^{k+1})) + \theta_k(H(x^{k+1};\xi^{k+1}) - \nabla^2 f(x^{k+1}))\|_F^3\big]\nonumber\\
& \overset{\eqref{cubic-function-1}}{\le} (1+c)\|(1-\theta_k)(M_k-\nabla^2 f(x^k)) + (1-\theta_k)(\nabla^2 f(x^k) - \nabla^2 f(x^{k+1}))\|^3_F\nonumber\\ 
&\qquad + 2(1 + c^{-1/2})\E_{\xi^{k+1}}\big[\|\theta_k(H(x^{k+1};\xi^{k+1}) - \nabla^2 f(x^{k+1}))\|_F^3\big]\nonumber\\
& \overset{\eqref{cubic-function-2}}{\le} (1+c)(1+2c)(1-\theta_k)^3\|M_k-\nabla^2 f(x^k)\|_F^3\nonumber \\
&\qquad + 2(1+c)(1 + c^{-1/2} + 2c^{-2})(1-\theta_k)^3\|\nabla^2 f(x^{k+1}) - \nabla^2 f(x^k)\|_F^3\nonumber\\
&\qquad + 2(1 + c^{-1/2})\theta_k^3\E_{\xi^{k+1}}\big[\|H(x^{k+1};\xi^{k+1}) - \nabla^2 f(x^{k+1})\|_F^3\big] \nonumber\\
& \le (1+c)(1+2c)(1-\theta_k)^3\|M_k-\nabla^2 f(x^k)\|_F^3 + 2(1+c)(1 + c^{-1/2} + 2c^{-2})(1-\theta_k)^3L_{F}^3\|x^{k+1}-x^k\|^3\nonumber\\
&\qquad + 2(1 + c^{-1/2})\sigma^3\theta_k^3,\label{raw-rec-bd}
\end{align} 
where the first inequality is due to \eqref{cubic-function-1} and $\E_{\xi^{k+1}}[H(x^{k+1};\xi^{k+1})]=\nabla^2 f(x^{k+1})$, the second inequality follows from \eqref{cubic-function-2}, and the last inequality follows from $\|\nabla^2 f(x^{k+1}) - \nabla^2 f(x^k)\|_F\le L_F\|x^{k+1}-x^k\|$ and $\E_{\xi^{k+1}}[\|H(x^{k+1};\xi^{k+1}) -\nabla^2 f(x^{k+1})\|^3_F]\le\sigma^3$. 

Letting $c = \theta_k/(2(1-\theta_k))$ in \eqref{raw-rec-bd} and using $\theta_k\in(0,1)$, we obtain that $c^{-1/2} = (2(1-\theta_k)/\theta_k)^{1/2} \le  \sqrt{2}\theta_k^{-1/2}$ and $c^{-2} = 4(1-\theta_k)^2/\theta_k^2 \le 4\theta_k^{-2}$. Combining these with \eqref{raw-rec-bd}, we obtain that 
\begin{align*}
&\E_{\xi^{k+1}}\big[\|M_{k+1} - \nabla^2 f(x^{k+1})\|_F^3\big]\le (1-\theta_k/2)(1-\theta_k)\|M_k-\nabla^2 f(x^k)\|_F^3\\
&\qquad + 2(1-\theta_k/2)(1-\theta_k)^2(1 + \sqrt{2}\theta_k^{-1/2} + 8\theta_k^{-2})L_{F}^3\|x^{k+1}-x^k\|^3 + 2(1 + \sqrt{2}\theta_k^{-1/2})\sigma^3\theta_k^3 \\
&\le (1-\theta_k)\|M_k-\nabla^2 f(x^k)\|_F^3 + 21 L_{F}^3\theta_k^{-2}\|x^{k+1}-x^k\|^3 + 5\sigma^3\theta_k^{5/2},
\end{align*}
where the second inequality is due to $\theta_k\in(0,1)$. Hence, the conclusion of this lemma holds as desired.
\end{proof}

The following lemma establishes a descent property for the potential sequence $\{\mathcal{P}_k\}$ defined below.

\begin{lemma}\label{th:potential-pm}
Suppose that Assumption \ref{asp:basic} holds. Let $\{(x^{k},M_k)\}$ be the sequence generated by Algorithm \ref{alg:unf-ssom-pm} with input parameters $\{(\eta_k,\theta_k)\}$. Assume that $\{\eta_k\}\subset(0,(2L)^{-1})$ and $\{\theta_k\}\subset(0,1)$, where $L$ is given in Assumption \ref{asp:basic}(b). Let $\{\mathcal{P}_k\}$ be defined in \eqref{def:pot} for $\{(x^k,M_k)\}$ and any positive sequence $\{p_k\}$ satisfying
\begin{align}\label{pkpm}
p_{k+1}= \frac{\theta_k^2}{378L_{F}^3\eta_k},\quad \frac{433\eta_k^2}{18}
+ (1-\theta_k)p_{k+1}\le p_k \quad \forall k\ge0,
\end{align}
where $L_F$ is given in Assumption \ref{asp:basic}(b). Then, it holds that
\begin{align}\label{eq:pk-pm}
\E_{\xi^{k+1}}[\mathcal{P}_{k+1}] \le \mathcal{P}_k - \eta_k^{1/2}\mu_{\eta_{k}}(x^{k+1})/18 + 55\eta_{k}^{1/2}\|g^k - \nabla f(x^k)\|^{3/2}/18 +  5\sigma^3\theta_k^{5/2}p_{k+1} \quad \forall k\ge 0,
\end{align}
where $\sigma$ is given in Assumption \ref{asp:basic}(c), and $\mu_{\eta}$ is defined in \eqref{eq:mu_eta}.
\end{lemma}

\begin{proof}
Fix any $k\ge0$. Notice that $\eta_k\in(0,(2L)^{-1})$. It follows from  \eqref{ineq:mu} and \eqref{lem:descent} with $(x^+,x,M,\eta)=(x^{k+1},x^k,M_k,\eta_k)$ that 
\begin{align}
&\mu_{\eta_k}(x^{k+1}) \le \eta_k^{-3/2}\|x^{k+1} - x^k\|^3 + \eta_k^{3/2}\|M_k - \nabla^2 f(x^k)\|_F^3 + \|g^k - \nabla f(x^k)\|^{3/2},\label{two-consec-pm}\\
&f(x^{k+1}) \le f(x^k) - (9\eta_k)^{-1}\|x^{k+1}-x^k\|^3 + 24\eta_{k}^2\|\nabla^2 f(x^k)-M_k\|_{F}^3 + 3\eta_{k}^{1/2} \|\nabla f(x^k)-g^k\|^{3/2}.\label{func-desc-pm}
\end{align}
Combining these with \eqref{def:pot} and \eqref{hessian-vr} , we obtain that
\begin{align*}
&\E_{\xi^{k+1}}[\mathcal{P}_{k+1}]  \overset{\eqref{def:pot}}{=} \E_{\xi^{k+1}}[f(x^{k+1}) + p_{k+1}\|M_{k+1}-\nabla^2 f(x^{k+1})\|_F^3] \\
& \overset{\eqref{hessian-vr}\eqref{func-desc-pm}}{\le } f(x^k) - ((9\eta_k)^{-1}-21L_{F}^3\theta_k^{-2}p_{k+1})\|x^{k+1}-x^k\|^3 \\
&\qquad + (24\eta_k^2 + (1-\theta_k)p_{k+1})\|M_k-\nabla^2f(x^k)\|_F^3 + 3\eta_{k}^{1/2}\|g^k-\nabla f(x^k)\|^{3/2} + 5 \sigma^3\theta_k^{5/2}p_{k+1}\\
& \overset{\eqref{two-consec-pm}}{\le } f(x^k) - \eta_k^{3/2}((9\eta_k)^{-1}-21L_{F}^3\theta_k^{-2}p_{k+1})\mu_{\eta_{k}}(x^{k+1})\\
& \qquad + ( \eta_k^3((9\eta_k)^{-1}-21L_{F}^3\theta_k^{-2}p_{k+1}) + 24\eta_k^2 
+ (1-\theta_k)p_{k+1})\|M_k-\nabla^2 f(x^k)\|_F^3 \\
& \qquad +  (3\eta_{k}^{1/2} + \eta_k^{3/2}((9\eta_k)^{-1}-21L_{F}^3\theta_k^{-2}p_{k+1}))\|g^k - \nabla f(x^k)\|^{3/2} + 5\sigma^3\theta_k^{5/2}p_{k+1}\\
& = f(x^k) - \eta_k^{1/2}\mu_{\eta_{k}}(x^{k+1})/18 + (433\eta_k^2/{18} 
+ (1-\theta_k)p_{k+1})\|M_k-\nabla^2 f(x^k)\|_F^3 \\
& \qquad + 55\eta_{k}^{1/2}\|g^k - \nabla f(x^k)\|^{3/2}/18 + 5\sigma^3\theta_k^{5/2}p_{k+1}\\
& \overset{\eqref{def:pot}}{\le} \mathcal{P}_k - \eta_k^{1/2}\mu_{\eta_{k}}(x^{k+1})/18 + 55\eta_{k}^{1/2}\|g^k - \nabla f(x^k)\|^{3/2}/18 + 5\sigma^3\theta_k^{5/2}p_{k+1},
\end{align*}
where the second equality is due to $p_{k+1}= \theta_k^2/(378L_{F}^3\eta_k)$, and the last inequality follows from \eqref{def:pot} and $433\eta_k^2/18
+ (1-\theta_k)p_{k+1}\le p_k$. The conclusion \eqref{eq:pk-pm} then follows from the above inequality.
\end{proof}

We are now ready to prove Theorem \ref{th:complexity-pm-c}.

\begin{proof}[{Proof of Theorem \ref{th:complexity-pm-c}}]
For convenience, let $\eta=1/(9K^{2/7})$. Then, we have $\eta_k=\eta$, $\theta_k=21L_F\eta$, and $\delta_k=9\eta^{2}$ for all $k\ge0$. Also, we define $p_k = 7\eta/(6L_F)$ for all $k\ge0$. Then, one can verify that \eqref{pkpm} holds for $\{(\eta_k,\theta_k,\delta_k)\}$ defined in \eqref{c-para-pm} and $\{p_k\}$ defined above. In addition, by \eqref{c-para-pm}, one has that $\{\eta_k\}\subset(0,(2L)^{-1})$ and $\{\theta_k\}\subset(0,1)$ holds for all $K\ge \max\{(2L/9)^{7/2},(7L_F/3)^{7/2},1\}$, Thus, Lemma \ref{th:potential-pm} holds for $\{(\eta_k,\theta_k,\delta_k)\}$ defined in \eqref{c-para-pm} and $\{p_k\}$ defined above. By the definition of $\{p_k\}$, $M_0=H(x^0;\xi^0)$, \eqref{ineq:G-var} and \eqref{ineq:H-Variance}, one has
\begin{align}
&\E[\mathcal{P}_0]=f(x^0) + p_0\E[\|M_0-\nabla^2 f(x^0)\|_F^3] \le f(x^0) + p_0\sigma^3 = f(x^0) + 7\eta\sigma^3/(6L_{F}),\label{ineq:upbd-exp0}\\
&\E[\mathcal{P}_K] = \E[f(x^K)+p_K\|M_K-\nabla^2 f(x^K)\|_F^3] \ge f_{\mathrm{low}}.\label{ineq:upbd-expK}
\end{align}
Notice that $\E_{\zeta^k}[\|g^k-\nabla f(x^k)\|^{3/2}]\le \delta_k^{3/2}$. Taking expectation of both sides of \eqref{eq:pk-pm} with respect to $\{\xi^i\}_{i=0}^{k+1}$ and $\{\zeta^i\}_{i=0}^k$, and substituting $\eta_k=\eta$, $\theta_k=21L_F\eta$, $\delta_k=9\eta^2$ and $p_k = 7\eta/(6L_F)$, we obtain that for all $k\ge0$,
\begin{align*}
\E[\mathcal{P}_{k+1}] \le \E[\mathcal{P}_{k}] - \eta^{1/2}\E[\mu_{\eta}(x^{k+1})]/18  + (11789\sigma^3L_{F}^{3/2}+83)\eta^{7/2}.
\end{align*}
Summing up this inequality over $k=0,\ldots,K-1$, and using \eqref{ineq:upbd-exp0} and \eqref{ineq:upbd-expK}, we can see that for all $K\ge \max\{(2L/9)^{7/2},(7L_F/3)^{7/2},1\}$,
\begin{align*}
f_{\mathrm{low}} & \overset{\eqref{ineq:upbd-exp0}}{\le}  \E[\mathcal{P}_K] \le    \E[\mathcal{P}_{0}] - (\eta^{1/2}/18)\sum_{k=0}^{K-1}\E[\mu_{\eta}(x^{k+1})] + (11789\sigma^3L_{F}^{3/2}+83)K\eta^{7/2}\\
&\overset{\eqref{ineq:upbd-expK}}{\le} f(x^0) + 7\eta\sigma^3/(6L_{F}) - (\eta^{1/2}/18)\sum_{k=0}^{K-1}\E[\mu_{\eta}(x^{k+1})] + (11789\sigma^3L_{F}^{3/2}+83)K\eta^{7/2}.
\end{align*}
Rearranging the terms of this inequality and using $\eta=1/(9K^{2/7})$, we obtain the following holds for all $K\ge \max\{(2L/9)^{7/2},(7L_F/3)^{7/2},1\}$,
\begin{align*}
\frac{1}{K}\sum_{k=0}^{K-1}\E[\mu_{\eta}(x^{k+1})] & \le 18\Big(\frac{f(x^0)-f_{\mathrm{low}} + 7\eta\sigma^3/(6L_{F})}{K\eta^{1/2}}+ (11789\sigma^3L_{F}^{3/2}+83)\eta^3\Big) \\
&\le 54(f(x^0)-f_{\mathrm{low}} + \sigma^3/(L_{F}^2) + L_{F}^{3/2}\sigma^3 + 1){K^{-6/7}} \overset{\eqref{def:Mc}}{=} M_{\mathrm{pm}} K^{-6/7}.
\end{align*}
Recall that $\iota_K$ is uniformly drawn from $\{1,\ldots,K\}$. This along with the above inequality implies that for all $K\ge \max\{(2L/9)^{7/2},(7L_F/3)^{7/2},1\}$,
\begin{align*}
\E[\mu_{\eta}(x^{\iota_K})] =  \frac{1}{K}\sum_{k=0}^{K-1}\E[\mu_{\eta}(x^{k+1})] \le M_{\mathrm{pm}}  K^{-6/7},    
\end{align*}
which along with the definition of $\mu_\eta$ in \eqref{eq:mu_eta} and the fact that $\eta=1/(9K^{2/7})$ implies the following holds for all $K\ge \max\{(2L/9)^{7/2},(7L_F/3)^{7/2},1\}$,
\begin{align*}
\E[\|\nabla f(x^{\iota_K})\|^{3/2}] \le 3M_{\mathrm{pm}}K^{-6/7},\quad \E[\lambda_{\min}(\nabla f(x^{\iota_K}))^3] \ge - 4M_{\mathrm{pm}}K^{-6/7}\eta^{-3/2} = -108M_{\mathrm{pm}}K^{-3/7}.
\end{align*}
In view of this, we can see that $x^{\iota_K}$ is an $(\epsilon_g,\epsilon_H)$-SSOSP of \eqref{ucpb} for all $K$ satisfying \eqref{stat-fix-pm}. Hence, the conclusion of this theorem holds as desired.
\end{proof}

\subsection{Proof of the main results in Section \ref{sec:scn-rm}}\label{subsec:proof-rm}

In this subsection, we present some technical lemmas and then use them to prove Theorem \ref{th:complexity-rm-c}. The following lemma gives the recurrence for the estimation error of the Hessian estimators $\{M_k\}$ generated by Algorithm \ref{alg:unf-ssom-rm}. 

\begin{lemma}\label{lem:rec-Hes-rm}
Suppose that Assumptions \ref{asp:basic} and \ref{asp:mean-squared} hold. Let $\{(x^k,M_k)\}$ be the sequence generated by Algorithm \ref{alg:unf-ssom-rm} with momentum parameters $\{\theta_k\}$. Then, it holds that for all $k\ge0$,
\begin{align}
\E_{\xi^{k+1}}[\|M^{k+1}-\nabla^2 f(x^{k+1})\|_F^3] \le (1-\theta_k)\|M_k-\nabla^2 f(x^k)\|_F^3 + 36(L_F^3 + L_H^3)\theta_k^{-1/2}\|x^{k+1}-x^k\|^3 + 36\theta_k^{5/2}\sigma^3, \label{hessian-vr-rm} 
\end{align}
where $L_F$ and $\sigma$ are given in Assumption \ref{asp:basic}, and $L_H$ is given in Assumption \ref{asp:mean-squared}.
\end{lemma}

\begin{proof}
Fix any $k\ge0$. It follows from \eqref{update-Mk-rm} that 
\begin{align}
 M_{k+1} - \nabla^2 f(x^{k+1}) &  \overset{\eqref{update-Mk-rm}}{=} (1-\theta_k)(M_k - \nabla^2 f(x^k)) + H(x^{k+1};\xi^{k+1}) - \nabla^2 f(x^{k+1})\nonumber\\
 &\quad \ + (1-\theta_k)(\nabla^2 f(x^k) - H(x^k;\xi^{k+1}))\label{a-identity-rm}
\end{align}
Observe from Assumptions \ref{asp:basic} and \ref{asp:mean-squared} that $\|\nabla^2 f(x^{k+1}) - \nabla^2 f(x^k)\|_F\le L_F\|x^{k+1}-x^k\|$, $\E_{\xi^{k+1}}[H(x^{k+1};\xi^{k+1})]=\nabla^2 f(x^{k+1})$, $\E_{\xi^{k+1}}[\|H(x^{k+1};\xi^{k+1}) -\nabla^2 f(x^{k+1})\|^3_F]\le\sigma^3$, and $\E_{\xi^{k+1}}[\|H(x^{k+1};\xi^{k+1}) -H(x^{k};\xi^{k+1})\|^3_F]\le L_H^3\|x^{k+1} - x^k\|_F^3$. Using these, \eqref{cubic-function-1}, and \eqref{a-identity-rm}, we obtain that 
\begin{align}
&\E_{\xi^{k+1}}[\|M_{k+1} - \nabla^2 f(x^{k+1})\|_F^3] \nonumber\\
&\overset{\eqref{a-identity-rm}}{=}\E_{\xi^{k+1}}[\|(1-\theta_k)(M_k - \nabla^2 f(x^k)) + H(x^{k+1};\xi^{k+1}) - \nabla^2 f(x^{k+1})+ (1-\theta_k)(\nabla^2 f(x^k) - H(x^k;\xi^{k+1}))\|_F^3]\nonumber\\
& \overset{\eqref{cubic-function-1}}{\le}(1+c)\|(1-\theta_k)(M_k - \nabla^2 f(x^k))\|_F^3 \nonumber \\
& \qquad + 2(1 + c^{-1/2})\E_{\xi^{k+1}}\|H(x^{k+1};\xi^{k+1}) - \nabla^2 f(x^{k+1}) + (1-\theta_k)(\nabla^2 f(x^k) - H(x^k;\xi^{k+1}))\|_F^3\nonumber\\
& = (1+c)(1-\theta_k)^3\|M_k-\nabla^2 f(x^k)\|_F^3 + 2(1 + c^{-1/2})\E_{\xi^{k+1}}\|H(x^{k+1};\xi^{k+1})-H(x^{k};\xi^{k+1}) \nonumber\\ 
& \qquad + \nabla^2 f(x^k) - \nabla^2 f(x^{k+1}) -\theta_k(\nabla^2 f(x^k) - H(x^k;\xi^{k+1}))\|_F^3 \nonumber\\
& \le (1+c)(1-\theta_k)^3\|M_k-\nabla^2 f(x^k)\|_F^3 + 18(1 + c^{-1/2})\E_{\xi^{k+1}}[\|H(x^{k+1};\xi^{k+1})-H(x^{k};\xi^{k+1})\|^3_F] \nonumber \\ 
& \quad + 18(1 + c^{-1/2})\|\nabla^2 f(x^k) - \nabla^2 f(x^{k+1})\|_F^3 + 18(1 + c^{-1/2})\theta_k^3\E_{\xi^{k+1}}[\|\nabla^2 f(x^k) - H(x^k;\xi^{k+1})\|_F^3]\nonumber\\
& \le (1+c)(1-\theta_k)^3\|M_k-\nabla^2 f(x^k)\|_F^3 + 18(1 + c^{-1/2})(L_F^3+ L_H^3)\|x^k-x^{k+1}\|^3 + 18\sigma^3(1 + c^{-1/2})\theta_k^3, \label{a-ineq-rm}
\end{align} 
where the first inequality follows from \eqref{cubic-function-1} and $\mathbb{E}_{\xi^{k+1}}[H(x^{k+1};\xi^{k+1})] = \nabla^2 f(x^{k+1})$, the second inequality is due to $\|A+B+C\|_F^3\le 9(\|A\|_F^3+\|B\|_F^3+\|C\|_F^3)$ for all $A,B,C\in\R^{n\times n}$, and the last inequality follows from $\|\nabla^2 f(x^{k+1}) - \nabla^2 f(x^k)\|_F\le L_F\|x^{k+1}-x^k\|$, $\E_{\xi^{k+1}}[\|H(x^{k+1};\xi^{k+1}) -\nabla^2 f(x^{k+1})\|^3_F]\le\sigma^3$, and $\E_{\xi^{k+1}}[\|H(x^{k+1};\xi^{k+1}) -H(x^{k};\xi^{k+1})\|^3_F]\le L_H^3\|x^{k+1} - x^k\|_F^3$.

Letting $c = \theta_k/(1-\theta_k)$ in \eqref{a-ineq-rm}, and using $\theta_{k}\in(0,1)$, we obtain $c^{-1/2}=(1-\theta_k)^{1/2}\theta_k^{-1/2} \le\theta_k^{-1/2}$. Combining this with \eqref{a-ineq-rm}, we obtain that 
\begin{align*}
\E_{\xi^{k+1}}[\|M_{k+1} - \nabla^2 f(x^{k+1})\|_F^3]&\le (1-\theta_k)^2\|M_k-\nabla^2 f(x^k)\|_F^3 \\
&\qquad + {18(L_F^3 + L_H^3)}{(1+\theta_k^{-1/2})}\|x^{k+1}-x^k\|^3 +  18\sigma^3(1+\theta_k^{-1/2})\theta_k^3,    
\end{align*}
which along with $\theta_k\in(0,1)$ implies that \eqref{hessian-vr-rm} holds as desired.
\end{proof}

The following lemma establishes a descent property for the potential sequence $\{\mathcal{P}_k\}$ defined below.

\begin{lemma}
\label{th:potential-rm}
Suppose that Assumptions \ref{asp:basic} and \ref{asp:mean-squared} hold. Let $\{(x^{k},M_k)\}$ be the sequence generated by Algorithm \ref{alg:unf-ssom-rm} with input parameters $\{(\eta_k,\theta_k)\}$. Assume that $\{\eta_k\}\subset(0,(2L)^{-1})$ and $\{\theta_k\}\subset(0,1)$, where $L$ is given in Assumption \ref{asp:basic}(b). Let $\{\mathcal{P}_k\}$ be defined in \eqref{def:pot} for $\{(x^k,M_k)\}$ and any positive sequence $\{p_k\}$ satisfying
\begin{align}\label{pkrm}
p_{k+1}  = \frac{\theta_k^{1/2}}{648(L_F^3 + L_H^3)\eta_k},\quad
\frac{433\eta_{k}^2}{18} + (1-\theta_k)p_{k+1} \le p_{k} \quad \forall k\ge 0,
\end{align}
where $L_{F}$ is given in Assumption \ref{asp:basic}(b) and $L_H$ is given in Assumption \ref{asp:mean-squared}. Then, it holds that 
\begin{align}\label{eq:pk-rm}
\E_{\xi^{k+1}}[\mathcal{P}_{k+1}]  \le \mathcal{P}_k - \eta_k^{1/2}\mu_{\eta_{k}}(x^{k+1})/18 + 55\eta_{k}^{1/2}\|g^k-\nabla f(x^k)\|^{3/2}/18 + 36\theta_k^{5/2}p_{k+1}\sigma^3  \quad \forall k\ge0,
\end{align}
where $\sigma$ is given in Assumption \ref{asp:basic}, and $\mu_{\eta}(x)$ is defined in \eqref{eq:mu_eta}.
\end{lemma}

\begin{proof}
Fix any $k\ge0$. Notice that $\eta_k\in(0,(2L)^{-1})$. It follows from \eqref{ineq:mu} and \eqref{lem:descent} with $(x^+,x,M,\eta)=(x^{k+1},x^k,M_k,\eta_k)$ that
\begin{align}
&\mu_{\eta_k}(x^{k+1}) \le \eta_k^{-3/2}\|x^{k+1} - x^k\|^3 + \eta_k^{3/2}\|M_k - \nabla^2 f(x^k)\|_F^3 + \|g^k - \nabla f(x^k)\|^{3/2},\label{two-consec-rm}\\
&f(x^{k+1}) \le f(x^k) - (9\eta_k)^{-1}\|x^{k+1}-x^k\|^3 + 24\eta_{k}^2\|\nabla^2 f(x^k)-M_k\|_{F}^3 + 3\eta_{k}^{1/2} \|\nabla f(x^k)-g^k\|^{3/2}.\label{func-desc-rm}
\end{align}
Combining these with \eqref{def:pot} and \eqref{hessian-vr-rm}, we obtain that
\begin{align*}
&\E_{\xi^{k+1}}[\mathcal{P}_{k+1}]  \overset{\eqref{def:pot}}{=} \E_{\xi^{k+1}}[f(x^{k+1}) + p_{k+1}\|M_{k+1}-\nabla^2 f(x^{k+1})\|_F^3] \\
& \overset{\eqref{hessian-vr-rm}\eqref{func-desc-rm}}{\le } f(x^k) - ((9\eta_k)^{-1}-36(L_F^3 + L_H^3)\theta_k^{-1/2}p_{k+1})\|x^{k+1}-x^k\|^3 \\
&\qquad + (24\eta_k^2 + (1-\theta_k)p_{k+1})\|M_k-\nabla^2f(x^k)\|_F^3 + 3\eta_{k}^{1/2}\|g^k-\nabla f(x^k)\|^{3/2} + 36\sigma^3\theta_k^{5/2}p_{k+1} \\
& \overset{\eqref{two-consec-rm}}{\le} f(x^k) - \eta_k^{3/2}((9\eta_k)^{-1}-36(L_F^3 + L_H^3)\theta_k^{-1/2}p_{k+1})\mu_{\eta_{k}}(x^{k+1})\\
& \quad + ( \eta_k^3((9\eta_k)^{-1}-36(L_F^3 + L_H^3)\theta_k^{-1/2}p_{k+1}) + 24\eta_k^2 
+ (1-\theta_k)p_{k+1})\|M_k-\nabla^2 f(x^k)\|_F^3\\
& \qquad + (3\eta_{k}^{1/2}+\eta_k^{3/2}((9\eta_k)^{-1}-36(L_F^3 + L_H^3)\theta_k^{-1/2}p_{k+1}))\|g^k-\nabla f(x^k)\|^{3/2} + 36\sigma^3\theta_k^{5/2}p_{k+1} \\
& = f(x^k) - \eta_k^{1/2}\mu_{\eta_{k}}(x^{k+1})/18 + (433\eta_k^2/{18} 
+ (1-\theta_k)p_{k+1})\|M_k-\nabla^2 f(x^k)\|_F^3\\
& \qquad + 55\eta_{k}^{1/2}\|g^k-\nabla f(x^k)\|^{3/2}/18 + 36\sigma^3\theta_k^{5/2}p_{k+1}\\
& \overset{\eqref{def:pot}}{\le} \mathcal{P}_k - \eta_k^{1/2}\mu_{\eta_{k}}(x^{k+1})/18 + 55\eta_{k}^{1/2}\|g^k-\nabla f(x^k)\|^{3/2}/18 + 36\sigma^3\theta_k^{5/2}p_{k+1},
\end{align*}
where the second equality is due to $p_{k+1}= \theta_k^{1/2}/(648(L_{F}^3+L_{H}^3)\eta_k)$, and the last inequality follows from \eqref{def:pot} and $433\eta_k^2/18
+ (1-\theta_k)p_{k+1}\le p_k$. The conclusion \eqref{eq:pk-rm} then follows from the above inequality.
\end{proof}

We now provide a proof of Theorem \ref{th:complexity-rm-c}.

\begin{proof}[{Proof of Theorem \ref{th:complexity-rm-c}}]
For convenience, let $\eta=1/(17K^{1/5})$. Then, we have $\eta_k=\eta$, $\theta_k=625(L_{F}^3 + L_H^3)^{2/3}\eta^2$ and $\delta_k=289\eta^3$ for all $k\ge0$. In addition, we define $p_k = 625^{1/2}/(648(L_{F}^3 + L_H^3)^{2/3})$ for all $k\ge0$. Then, one can verify that \eqref{pkrm} holds for $\{(\eta_k,\theta_k,\delta_k)\}$ defined in \eqref{c-para-rm} and $\{p_k\}$ defined above. In addition, by \eqref{c-para-rm}, one has that $\{\eta_k\}\subset(0,(2L)^{-1})$ and $\{\theta_k\}\subset(0,1)$ holds for all $K\ge \max\{(2L/17)^{5}, 7(L_{F}^3 + L_H^3)^{5/3},1\}$, Thus, Lemma \ref{th:potential-rm} holds for $\{(\eta_k,\theta_k,\delta_k)\}$ defined in \eqref{c-para-rm} and $\{p_k\}$ defined above. By the definition of $\{p_k\}$, $M_0=H(x^0;\xi^0)$, and \eqref{ineq:H-Variance}, one has
\begin{align}
&\E[\mathcal{P}_0]=f(x^0) + p_0\E[\|M_0-\nabla^2 f(x^0)\|_F^3]\le f(x^0) + p_0\sigma^3 \le f(x^0) + \sigma^3/(L_{F}^3 + L_H^3)^{2/3},\label{ineq:upbd-exp0-rm}\\
&\E[\mathcal{P}_K] = \E[f(x^K)+p_K\|M_K-\nabla^2 f(x^K)\|_F^3] \ge f_{\mathrm{low}}.\label{ineq:upbd-expK-rm}
\end{align}
Notice that $\E_{\zeta^k}[\|g^k-\nabla f(x^k)\|^{3/2}]\le \delta_k^{3/2}$. Taking expectation of both sides of \eqref{eq:pk-rm} with respect to $\{\xi^i\}_{i=0}^{k+1}$ and $\{\zeta^i\}_{i=0}^k$, and substituting $\eta_k=\eta$, $\theta_k=625(L_{F}^3 + L_H^3)^{2/3}\eta^2$, and $p_k = 625^{1/2}/(648(L_{F}^3 + L_H^3)^{2/3})$, we obtain that for all $k\ge0$,
\begin{align*}
\E[\mathcal{P}_{k+1}] \le \E[\mathcal{P}_{k}] - \eta^{1/2}\E[\mu_{\eta}(x^{k+1})]/18 + \Big(\frac{55}{18}\cdot 17^3 + \frac{36\cdot625^3}{648}\sigma^3(L_{F}^3+L_H^3)\Big)\eta^5.
\end{align*}
Summing this inequality over $k=0,\ldots,K-1$, and using \eqref{ineq:upbd-exp0-rm} and \eqref{ineq:upbd-expK-rm}, it follows that for all $K\ge \max\{(2L/17)^{5},7(L_{F}^3 + L_H^3)^{5/3},1\}$,
\begin{align*}
f_{\mathrm{low}} & \overset{\eqref{ineq:upbd-expK-rm}}{\le}  \E[\mathcal{P}_K] \le \E[\mathcal{P}_{0}] - (\eta^{1/2}/18)\sum_{k=0}^{K-1}\E[\mu_{\eta}(x^{k+1})] + \Big(\frac{55}{18}\cdot 17^3 + \frac{36\cdot625^3}{648}\sigma^3(L_{F}^3+L_H^3)\Big)K\eta^5\\
&\overset{\eqref{ineq:upbd-exp0-rm}}{\le} f(x^0) + \sigma^3/(L_{F}^3 + L_H^3)^{2/3} - (\eta^{1/2}/18)\sum_{k=0}^{K-1}\E[\mu_{\eta}(x^{k+1})] + \Big(\frac{55}{18}\cdot 17^3 + \frac{36\cdot625^3}{648}\sigma^3(L_{F}^3+L_H^3)\Big)K\eta^5.
\end{align*}
Rearranging the terms of this inequality and using the definition of $\eta=1/(17K^{1/5})$, we obtain that for all $K\ge \max\{(2L/17)^{5},7(L_{F}^3 + L_H^3)^{5/3},1\}$,
\begin{align*}
\frac{1}{K}\sum_{k=0}^{K-1}\E[\mu_{\eta}(x^{k+1})] & \le 18\Big(\frac{f(x^0)-f_{\mathrm{low}} + \sigma^3/(L_{F}^3 + L_H^3)^{2/3}}{K\eta^{1/2}}+ \Big(\frac{55}{18}\cdot 17^3 + \frac{36\cdot625^3}{648}\sigma^3(L_{F}^3+L_H^3)\Big)\eta^{9/2}\Big) \\
&\le 75(f(x^0)-f_{\mathrm{low}} + \sigma^3/(L_{F}^3 + L_H^3)^{2/3} + (L_{F}^3+L_H^3)\sigma^3 + 1) K^{-9/10} \overset{\eqref{def:Mct}}{=} M_{\mathrm{rm}} K^{-9/10}.
\end{align*}
Recall that $\iota_K$ is uniformly drawn from $\{1,\ldots,K\}$. This along with the above inequality implies that for all $K\ge \max\{(2L/17)^{5},7(L_{F}^3 + L_H^3)^{5/3},1\}$,
\begin{align*}
\E[\mu_{\eta}(x^{\iota_K})] =  \frac{1}{K}\sum_{k=0}^{K-1}\E[\mu_{\eta}(x^{k+1})] \le M_{\mathrm{rm}} K^{-9/10},    
\end{align*}
which along with the definition of $\mu_\eta$ in \eqref{eq:mu_eta} and $\eta=1/(17K^{1/5})$ implies that for all $K\ge \max\{(2L/17)^{5},7(L_{F}^3 + L_H^3)^{5/3},1\}$,
\begin{align*}
\E[\|\nabla f(x^{\iota_K})\|^{3/2}] \le 3M_{\mathrm{rm}}K^{-9/10},\quad \E[\lambda_{\min}(\nabla f(x^{\iota_K}))^3] \ge - 4M_{\mathrm{rm}}K^{-9/10}\eta^{-3/2} = -281M_{\mathrm{rm}}K^{-3/5}.
\end{align*}
In view of these, we can see that $x^{\iota_K}$ is an $(\epsilon_g,\epsilon_H)$-SSOSP of \eqref{ucpb} for all $K$ satisfying \eqref{stat-fix-rm}. Hence, the conclusion of this theorem holds as desired.
\end{proof}

\section*{Acknowledgment} The work of Yiming Yang was partly supported by Postgraduate Scientific Research Innovation Project of Hunan Province (Grant: CX20230617). The work of Zheng Peng was partly supported by the Major Research Plan of National Natural Science Foundation of China (Grant: 92473208), the Key Program of National Natural Science of China (Grant:12331011), and the Innovative Research Group Project of Natural Science Foundation of Hunan Province (Grant: 2024JJ1008). The work of Xiao Wang was partly supported by National Natural Science Foundation of China (Grant: 12271278). The work of Chuan He was partially supported by the Wallenberg AI, Autonomous Systems and Software Program (WASP) funded by the Knut and
Alice Wallenberg Foundation.

\bibliographystyle{abbrv}
\bibliography{ref}

\end{document}